\documentclass[a4paper]{elsarticle}
\usepackage{lineno,hyperref}
\modulolinenumbers[5]

\usepackage{enumerate}
\usepackage{bm}
\usepackage{multirow}
\usepackage{geometry}
\usepackage{textcomp}
\usepackage{graphicx}
\usepackage{amsmath}
\usepackage{amssymb}
\usepackage{amsthm}
\usepackage{latexsym}
\usepackage{booktabs}
\usepackage{float}
\usepackage{makecell} 
\usepackage{array}
\usepackage{amscd}
\usepackage{tikz}
\usepackage{url}
\usepackage{hyperref}
\usetikzlibrary{trees}

\usepackage{geometry}
\newtheorem{Definition}{Definition}[section]
\newtheorem{thm}{Theorem}[section]
\newtheorem{pro}{Proposition}[section]
\newtheorem{lem}{Lemma}[section]
\newtheorem{cor}{Corollary}[section]

\theoremstyle{remark}
\newtheorem{exam}{Example}
\newtheorem{remark}{Remark}[section]
\makeatletter
\begin{document}
\begin{frontmatter}
\title{On two problems about isogenies of elliptic curves\\ over finite fields}
\author[CAS,UCAS]{Lixia Luo}\ead{luolixia@amss.ac.cn}
\author[CAS,UCAS]{Guanju Xiao}\ead{gjXiao@amss.ac.cn}
\author[CAS,UCAS]{Yingpu Deng}
\ead{dengyp@amss.ac.cn}
\address[CAS]{Key Laboratory of Mathematics Mechanization, NCMIS, Academy of Mathematics and Systems Science, Chinese Academy of Sciences, Beijing 100190, People's Republic of China}
\address[UCAS]{School of Mathematical Sciences, University of Chinese Academy of Sciences, Beijing 100049, People's Republic of China}

\begin{abstract}
	 Isogenies occur throughout the theory of elliptic curves. Recently, the cryptographic protocols based on isogenies are considered as candidates of quantum-resistant cryptographic protocols. Given two elliptic curves $E_1, E_2$ defined over a finite field $k$ with the same trace, there is a nonconstant isogeny $\beta$ from $E_2$ to $E_1$ defined over $k$. This study  gives out the index of $\rm{Hom}_{\it k}(\it E_{\rm 1},E_{\rm 2})\beta$ as a left ideal in $\rm{End}_{\it k}(\it E_{\rm 2})$ and figures out the correspondence between isogenies and kernel ideals. In addition, some results about the non-trivial minimal degree of isogenies between the two elliptic curves are also provided.
\end{abstract}
\end{frontmatter}

\bibliographystyle{plain}

\section{Introduction}
Isogenies play an important part in  the theory of elliptic curves. A recent research area is cryptographic protocols based on the difficulty of constructing isogenies between elliptic curves over finite fields\cite{ azarderakhsh2017supersingular,charles2009cry,galbraith2017identification}. These cryptographic protocols are supposed to resist the quantum computations. To get more facts about isogenies, this paper concerns two problems related to isogenies.

Let $F$ be a perfect field, and $E_1, E_2$ be elliptic curves defined over $F$, it had been proved that $\rm{Hom}(\it E_{\rm 1},\it E_{\rm 2})$ is a free $\mathbb Z$-module of rank at most $4$ \cite[Corollary \uppercase\expandafter{\romannumeral3}.7.5]{silverman2009arithmetic}. Further, the possible ranks of $\rm {End}(\it E_{\rm 1})$ (or $\rm {End}(\it E_{\rm 2})$) are $1,2,4$. The possible result that rank$_{\mathbb{Z}}$ $\rm{Hom}(\it E_{\rm 1},\it E_{\rm 2})=\rm3$ is proved to be negative. If there is a nonconstant isogeny $\beta$ from $E_2$ to $E_1$, $\rm{Hom}(\it E_{\rm 1},\it E_{\rm 2})\beta$ is a left ideal of $\rm {End}(\it E_{\rm 2})$, then $\rm{rank}_{\mathbb{Z}}\rm{Hom}(\it E_{\rm 1},\it E_{\rm 2})=\rm{rank}_{\mathbb{Z}}\rm {End}(\it E_{\rm 2})$. The index of $\rm{Hom}(\it E_{\rm 1},\it E_{\rm 2})\beta$ in $\rm {End}(\it E_{\rm 2})$ is finite, but the exact result needs to be identified. We assume that the characteristic of $F$ is not $0$ and for the cases $\rm{char}(\it F) = \rm 0$, we discuss them in Appendix A.

In Waterhouse's thesis\cite{waterhouse1969abe}, he introduced the concept of kernel ideals and proved that for any elliptic curve $E$ over a finite field, the left ideals of its endomorphism ring are all kernel ideals. Every such left ideal can induce an isogeny from $E$, and ideal multiplication corresponds to isogeny composition.  Kohel proved the correspondence between the invertible ideals and the isogenies between two ordinary elliptic curves with the same endomorphism type and
used the invertible ideals of imaginary quadratic orders to compute the endomorphism types of ordinary elliptic curves over finite fields\cite{kohel1996endomorphism}. In this paper, we explore the index of $\rm{Hom}_{\it k}(\it E_{\rm 1},E_{\rm 2})\beta$ as a left ideal in $\rm{End}_{\it k}(\it E_{\rm 2})$ for any nonconstant isogeny $\beta$ from $E_2$ to $E_1$ defined over a finite field $k$ as the first problem.  For ordinary elliptic curves, not all isogenies can correspond to kernel ideals. We will give out a way to judge whether the isogenies correspond to kernel ideals in this case. Besides, it is natural for us to consider the non-trivial minimal degree of isogenies between any two elliptic curves over finite fields as the second problem.

The paper is organized as follows. In Section 2, we provide the preliminaries on elliptic curves, isogenies, endomorphism rings and kernel ideals. The answer to the first problem will be given out by Theorems 3.1 and 3.2 in Section 3. And the study of the second problem can be found in Section 4.

\section{Preliminaries}
\subsection{Elliptic curves and isogenies}
Let $k$ be a finite field of characteristic $p$ and let $E$ be an elliptic curve defined over $k$. $E$ can be written in a generalised Weierstrass equation
\begin{equation*}
  E:y^2+a_{\rm 1}xy+a_3y=x^3+a_{\rm 2}x^2+a_4x+a_6.
\end{equation*}
For simplicity, if $p>3$, up to $k$-isomorphism, $E$ can be written as the short Weierstrass form
\begin{equation*}
  E:y^2=x^3+Ax+B
\end{equation*}
with its discriminant $4A^3+27B^2$ not equal to 0.  Under this form the $j$-invariant of $E$ is defined as $1728(4A^3)/(4A^3+27B^2)$, which is an invariant up to $\bar{k}$-isomorphism. $E(k)$ is a finite abelian group with unity $\infty$ under the chord and tangent addition law \cite{washington2008elliptic}.

Let $E_{\rm 1}$ and $E_{\rm 2}$ be two elliptic curves defined over $k$. An isogeny  $\alpha$ from $E_{\rm 1}$ to $E_{\rm 2}$ is a morphism mapping $\infty$ to $\infty$. If $\alpha(E_{\rm 1})=\infty$, $\alpha$ is called constant and denoted by 0. Its degree is defined to be $0$. Otherwise, $\alpha$ can be represented as $(r_{\rm 1}(x),r_{\rm 2}(x)y)$ where $r_{\rm 1}(x)$, $r_{\rm 2}(x)$ are rational functions and we can write $r_{\rm 1}(x)=p(x)/q(x)$ with polynomials $p(x)$ and $q(x)$ that don't have a common factor and $q(x)\ne 0$. The degree of $\alpha$ is defined to be
\begin{equation*}
  \rm{deg}(\alpha)=\rm{max}\{\rm{deg}\it p(x),\rm{deg}\it q(x)\}.
\end{equation*} $\alpha$ is called separable if $(r_{\rm 1}(x))'\ne 0$. All nonconstant isogenies are surjective homomorphisms with finite kernels and $\vert \rm{ ker} (\alpha)\vert = \rm{deg}\alpha$ if $\alpha$ is separable \cite{washington2008elliptic}. According to \cite[Proposition \uppercase\expandafter{\romannumeral3}.4.12]{silverman2009arithmetic}, the nonconstant separable isogenies from a given elliptic curve can be distinguished by their kernels. For every nonconstant isogeny $\beta: E_{\rm 2}\to E_{\rm 1}$ of degree $m$, there is a unique dual isogeny $\widehat{\beta}: E_{\rm 1}\to E_{\rm 2}$ satisfying $\widehat{\beta}\beta=[m]_{E_2}$ and $\beta\widehat{\beta}=[m]_{E_1}$ where $[m]$ is the multiplication-by-m map of degree $m^2$ and $\rm{deg}\widehat{\beta}=\it m$ \cite[Theorem \uppercase\expandafter{\romannumeral3}.6.1]{silverman2009arithmetic}.
For the properties of dual isogenies, we can refer to  \cite[Theorem \uppercase\expandafter{\romannumeral3}.6.2]{silverman2009arithmetic}. All the isogenies we consider in the following are nonconstant.

Let $\rm{Hom}(\it{E}_{\rm{1}},\it{E}_{\rm{2}})$ denote the collection of all isogenies from $E_{\rm 1}$ to $E_{\rm 2}$. Similarly, $\rm{Hom}_{\it{k}}(\it{E}_{\rm{1}},\it{E}_{\rm{2}})$ can be defined for the isogenies defined over $k$. If $E_{\rm 1}=E_{\rm 2}$, the isogenies are called endomorphisms. All endomorphisms of an elliptic curve $E$ form a ring and it is called the endomorphism ring of $E$ and denoted by $\rm{End}(\it{E})$. The endomorphism ring defined over $k$ is denoted by  $\rm{End}_{\it{k}}(\it{E})$.

Let $q=\vert k\vert $ and $t=q+1-\vert E(k) \vert $. Then the Frobenius endomorphism $\phi_{q, E}:
(x,y)\mapsto(x^q,y^q)$ of $E$ related to $k$ has characteristic polynomial  $\mathit h_{q,E} :=x^2-tx+q$. $t$ is called the trace of $\phi_{q, E}$ or $E$ over $k$. By Hasse's theorem, we have $t^2-4q\leqslant 0$. Let $\pi_{q,E}$ be a root of $\mathit h_{q,E}$. If there is a nonconstant isogeny from $E_{\rm 1}$ to $E_{\rm 2}$ defined over $k$, $E_{\rm 1}$ and $E_{\rm 2}$ are called $k$-isogenous. It is well known that  $E_{\rm 1}$ and $E_{\rm 2}$ are $k$-isogenous if and only if $\mathit h_{q,E_{\rm 1}}=\mathit h_{q,E_{\rm 2}}$ if and only if $\vert E_{\rm 1}(k)\vert=\vert E_{\rm 2}(k)\vert$ \cite[Theorem 1]{tate1966end}. All the $k$-isogenous elliptic curves form an isogeny class over $k$.

\subsection{Endomorphism rings and kernel ideals}

An elliptic curves $E$ is called supersingular if $p\mid t$ and ordinary, otherwise. As we all know, $\rm{End}(\it{E})$ is isomorphic to an imaginary quadratic order containing $\pi_{q,E}$ when $E$ is ordinary or a maximal order of a definite quaternion algebra $B_{p,\infty}$ over $\mathbb Q$ ramified at $p$ and $\infty$  when $E$ is supersingular\cite[Theorem \uppercase\expandafter{\romannumeral5}.3.1]{silverman2009arithmetic}. According to Waterhouse's thesis\cite[Theorem 4.2]{waterhouse1969abe}, $\rm{End}_{\it{k}}(\it{E})$ is equal to $\rm{End}(\it{E})$ when $p\nmid t$ or $t= \pm 2\sqrt{q}$ and for other cases, $\rm{End}_{\it{k}}(\it{E})$ is isomorphic to an imaginary quadratic order which is ramified at $p$ containing $\pi_{q,E}$. Then $\rm{End}(\it{E})($$\rm{End}_{\it{k}}(\it{E}))\otimes \mathbb Q$ is isomorphic to an imaginary quadratic field or $B_{p,\infty}$. In general, we will denote the field or the algebra by $K$, and we call it endomorphism algebra. The isogenous elliptic curves have the same endomorphism algebra.

Any definite quaternion algebra over $\mathbb Q$ has a representation of the form $\left(\frac{a,b}{\mathbb Q}\right)=\mathbb Q +\mathbb Q i+\mathbb Q j+\mathbb Q k$ where $i^2=a, j^2=b, k=ij=-ji$. Define the conjugation of the elements of $\left(\frac{a,b}{\mathbb Q}\right)$ given by
\begin{equation*}
  \alpha=x+yi+zj+wk \mapsto \bar{\alpha}=x-yi-zj-wk,
\end{equation*}
from which the reduced trace and norm take the form
\begin{equation*}
  \rm{Trd}(\alpha)=\alpha+\bar{\alpha}=2\it x\    \rm{and}\    \rm{Nrd}(\alpha)=\alpha\bar{\alpha}=\it x^{\rm 2}-ay^{\rm2}-bz^{\rm2}+abw^{\rm2}.
\end{equation*}
For simplicity, we use the same symbols for the trace and norm of imaginary quadratic fields. Let $i_{E}: K \to $ $\rm{End}(\it{E})\otimes \mathbb Q$ be an isomorphism, then every endomorphism $\alpha$ and its dual satisfy the characteristic polynomial $x^2-\rm{Trd}$$(i_{E}^{-1}(\alpha))x+\rm{Nrd}$$(i_{E}^{-1}(\alpha))$.
Under this isomorphism, the dual of every endomorphism corresponds to the conjugate of the corresponding element in the imaginary quadratic field or the conjugation of the corresponding element in $B_{p,\infty}$.  And the degree corresponds to the field norm of the imaginary quadratic field or the reduced norm of $B_{p,\infty}$.

Let $\beta:E_2\to E_1$ be an isogeny defined over $k$. Let $i_{E_{\rm 2}}: K \to $ $\rm{End}_{\it{k}}(\it{E}_{\rm{2}})\otimes \mathbb Q$ be an isomorphism and fix $i_{E_{\rm 2}}(\pi_{q,E_{\rm 2}})=\phi_{q,E_{\rm 2}}$, then we can induce
\begin{equation*}
\begin{split}
     i_{E_{\rm 1}}: K & \to  \rm{End}_{\it{k}}(\it{E}_{\rm{1}})\otimes \mathbb Q  \\
     \zeta & \mapsto \frac{1}{\rm{deg}(\beta)} \beta \it{i}_{E_{\rm 2}}(\zeta)\widehat {\beta}
  \end{split}
\end{equation*}
where $\widehat {\beta}$ is the dual isogeny of $\beta$ and $i_{E_{\rm 1}}^{-1}(\phi_{q,E_{\rm 1}})=i_{E_{\rm 2}}^{-1}(\phi_{q,E_{\rm 2}})$. Hence under these two isomorphisms, we have
\begin{equation*}
i_{E_{\rm 2}}^{-1}(\rm{Hom}_{\it{k}}(\it{E}_{\rm{1}},\it{E}_{\rm{2}})\beta)=i_{E_{\rm 1}}^{-\rm1}(\beta \rm{Hom}_{\it{k}}(\it{E}_{\rm{1}},\it{E}_{\rm{2}})).
\end{equation*}
In the following, for simplicity, we take $\rm{End}_{\it{k}}(\it{E})$ as an order of $K$ and take $\rm{Hom}_{\it{k}}(\it{E}_{\rm{1}},\it{E}_{\rm{2}})\beta$ as $i_{E_{\rm 2}}^{-1}($$\rm{Hom}_{\it{k}}(\it{E}_{\rm{1}},\it{E}_{\rm{2}})\beta)$, then $\rm{Hom}_{\it{k}}(\it{E}_{\rm{1}},\it{E}_{\rm{2}})\beta=\beta $$\rm{Hom}_{\it{k}}(\it{E}_{\rm{1}},\it{E}_{\rm{2}}))$ is a left ideal of  $\rm{End}_{\it{k}}(\it{E}_{\rm{2}})$  and a right ideal of  $\rm{End}_{\it{k}}(\it{E}_{\rm{1}})$.  Since $i_{E_{\rm 2}}^{-1}(\widehat \alpha)$ is equal to $\overline{i_{E_{\rm 2}}^{-1}(\alpha)}$, then \begin{equation*}
  i_{E_{\rm 1}}^{-1}(\rm{Hom}_{\it{k}}(\it \it E_{\rm 2},\it \it E_{\rm 1})\widehat{\beta})= i_{E_{\rm 2}}^{-\rm1}(\widehat{\beta}\rm{Hom}_{\it{k}}(\it \it E_{\rm 2},\it \it E_{\rm 1}))=\overline{i_{E_{\rm 2}}^{-\rm1}(\rm{Hom}_{\it{k}}(\it \it E_{\rm 1},\it \it E_{\rm 2})\beta)}.
 \end{equation*}

We now investigate kernel ideals and their relations with isogenies. The main references are Waterhouse's thesis\cite{waterhouse1969abe} and \cite[Chapter 42.2]{voight2019quaternion}. For an isogeny $\beta$ from an elliptic curve $E$ defined over $k$, let $\rm{H}(\beta)$ be its kernel as a finite group scheme over $k$. Define the rank of a finite group scheme $H$ as dim$_{k}k[H]$. And rank H$(\beta)$=deg$\beta$. Give an non-zero ideal $J$ of $ \rm{End}_{\it{k}}(\it{E})$, let $\rm{H}(\it{J})$ be the scheme-theoretic intersection of the kernels of all elements of $J$. Let $ \rm{I}(\it H):=\{\alpha \in \rm{End}_{\it{k}}(\it{E}) :  \alpha(\it H)=\infty\}$. $J$ is called kernel ideal if $J = \rm{I}(\rm{H}(\it{J}))$. In fact, all the ideals considered in this paper are kernel ideals due to \cite{waterhouse1969abe}.

 \begin{pro}
   Given an elliptic curve $E$ over finite field $k$.
   \begin{enumerate}[1)]
     \item If $\rm{End}(\it E)$ is isomorphic to a maximal order of quaternion algebra, then every ideal $I$ of $\rm{End}(\it E)$ is a kernel ideal, and $\rm{rank}\ H(\it I)$ is equal to the reduced norm $\rm{Nr}$$(I)$.
     \item If $\rm{End}(\it E)$ is isomorphic to an imaginary quadratic order, then every ideal $I$ of $\rm{End}(\it E)$ is a kernel ideal, and if $I$ is invertible, then $\rm{rank}\ H(\it I)$ is equal to the norm of $I$.

   \end{enumerate}

 \end{pro}
 \begin{proof}
   See\cite[Theorems 3.15 and 4.5]{waterhouse1969abe}. If $\rm{End}(\it E)$ is not maximal, for every invertible ideal $I$, it is similar to the proof of Theorem 3.15 where Waterhouse just proved for the case when $\rm{End}(\it E)$ is maximal that $\rm{rank}\ H(\it I)$ is equal to the norm of $I$.
 \end{proof}

 Waterhouse also proved that ideal multiplication corresponds to composition of isogenies.
 \begin{pro}
 Let $I$ be a left ideal of the endomorphism ring End$_k(E)$ , $\varphi_I:E\to E/\rm{H}(\it I)=E'$ be the isogeny with kernel H$(I)$. Let $J$ be a left ideal of End$_k(E')$,  $\varphi_J:E'\to E'/\rm{H}(\it J)=E''$. Then $\varphi_J\varphi_I$ is the isogeny from $E$ to $E''$ with kernel H$(IJ)$.
 \end{pro}
 \begin{proof}
   See\cite[Proposition 3.12]{waterhouse1969abe}.
 \end{proof}
 If an isogeny from $E$ has kernel $\rm{H}(\it{J})$ for some ideal $J$ of $\rm{End}_{\it{k}}(\it{E})$ or $\rm{End}(\it{E})$, then we say the isogeny are corresponding to $J$ and $J$ is the kernel ideal of the isogeny. In Waterhouse's thesis, he mentioned that not every finite subgroup of $E$ has the form $\rm{H}(\it{J})$. In fact, we will give out a way to judge whether the isogenies can correspond to ideals in Section 3.

 Let $\beta: E_2\to E_1$ be an isogeny defined over $k$. If deg$(\beta)$ is prime to $p$, then H$(\beta)$=$\rm{ker}(\beta)$ as a subgroup of $E_2(\bar{k})$ which is closed under the action of the Galois group of $\bar{k}$ over $k$. Otherwise, we do some discussion following.

  Given an elliptic curve $E$ defined over $k=\mathbb F_{p^n}$, the Frobenius map $\phi_{p,E}: E\to E^{(p)}$ with $(x,y)\to(x^p,y^p)$ is defined over $k$. And if $E: y^2=x^3+Ax+B$, then $E^{(p)}: y^2=x^3+A^{p}x+B^{p}$. Although \rm{ker}$(\phi_{\it p,E})$=$\{\infty\}$, H$(\phi_{\it p,E})$ is a finite group scheme of rank $p$ and $\rm{Hom(\it E^{(p)},E})\phi_{\it p,E}$ is the kernel ideal of $(\phi_{\it p,E})$. And
  $\phi_{p^i,E}: E\to E^{(p^i)}$ with $(x,y)\to(x^{p^i},y^{p^i})$ is equal to $\phi_{p,E^{(p^{i-1})}}\cdots \phi_{p,E^{(p)}}\phi_{p,E}$ where $i \in \mathbb Z_{>0}$ and $\phi_{p^n,E}$ is the Frobenius endomorphism of $E$ over $k$. If $\beta\in \rm{Hom}(\it{E}_{\rm{2}},\it{E}_{\rm{1}})$ is not separable,  by \cite[Corollarly \uppercase\expandafter{\romannumeral2}.2.12]{silverman2009arithmetic}, $\beta$ can be factored as
\begin{equation*}
E_{\rm 2}\stackrel{\phi_{q'}}\to E_{\rm 2}^{(q')}\stackrel{\lambda}\to  E_{\rm 1},
\end{equation*}
 where $\phi_{q'}$ is the $q'$th-power Frobenius map, $q'$ is the inseparable degree of $\beta$ and the map $\lambda$ is separable.

If $E$ is supersingular, then $\widehat{\phi _{p^n,E}}$ is purely inseparable.  Let $\beta$ be any isogeny from $E$ of degree $p^e$, then $\beta=\lambda \phi_{p^e, E}$, where $\lambda$ is an isomorphism. By Proposition 2.2, $\beta$ corresponds to the following ideal
\begin{equation*}
\rm{Hom}(\it E^{(p)},\it E)\phi_{p,E}\rm{Hom}(\it E^{(p^{\rm2})},\it E^{(p)})\phi_{p,E^{(p)}}\cdots
\rm{Hom}(\it E^{(p^{e})},\it E^{(p^{e-\rm1})})\phi_{p,E^{p^{(e-\rm1)}}} .
\end{equation*}
 Given a maximal order type containing some order $\mathcal O$, whose localization $\mathcal O_p$ has unique maximal two-sided ideal $\mathfrak{P}$. According to Deuring's correspondence between supersingular $j$-invariants and the types of maximal orders in $B_{p,\infty}$\cite[Chapter 10.2]{deuring1941constructing}\cite[Theorem 44]{kohel1996endomorphism}, if $\mathfrak{P}$ is principal, then there is only one $j$-invariant which is defined over $\mathbb F_p$ such that its endomorphism ring corresponds to this type; if $\mathfrak{P}$ is not principal, then there are just two $j$-invariants which are a conjugate pair over $\mathbb F_{p^2}$ such that their endomorphism rings correspond to this type. Hence, if $j$-invariant of $E$ is defined over $\mathbb F_p$, then $\rm{Hom}(\it E^{(p)},\it E)\phi_{p,E}$ is the unique principal prime ideal $\mathfrak{P}$ over $p$; if $j$-invariant of $E$ is not defined over $\mathbb F_p$, then $\rm{Hom}(\it E^{(p)},\it E)\phi_{p,E}$ is the unique two-sided prime ideal $\mathfrak{P}$ over $p$ where $\mathfrak{P}$ is not principal and $\mathfrak{P}^2=(p)$. Thus $\beta$ corresponds to $\mathfrak{P}^e$.

 If $E$ is ordinary, then  $\widehat{\phi _{p^n,E}}$ is separable and $\rm{ker} \widehat{\phi _{\it p^n,E}}=\it E[p^n]$ which is a cyclic group. Since the Kronecker symbol $(\frac{t^2-4q}{p})=1$, then $p$ splits in $\rm{End}(\it{E})$ and there are two prime ideals $\mathfrak P_{\rm 1},\mathfrak P_{\rm 2}$ of norm $p$. Let $\mathfrak P_{\rm 1}$ be the one such that $\mathfrak P_{\rm 1}^n=(\pi_{p^n, E})$ where $q=p^n$, then $\mathfrak P_{\rm 1}=\rm{Hom}(\it{E}^{(p)},E)\phi_{p,E}$ and the other prime ideal $\mathfrak P_{\rm 2}=\rm{Hom}(\it{E},E^{(p)})\widehat \phi_{\it p,E}=\rm{Hom}(\it E^{(\it{p}^{\it n-\rm1})},E)\widehat{\phi}_{\it p,E^{(\it p^{n-\rm 1})}}$. Hence $\mathfrak P_{\rm 1}$ is the kernel ideal corresponding to $\phi_{p,E}$ and $\mathfrak P_{\rm 2}$ is  the kernel ideal corresponding to $\widehat{\phi}_{p,E^{(p^{n-1})}}$.
 Let $\beta$ be an isogeny from $E$ with degree $p^e$ and inseparable degree $p^{e_{\rm 1}}$, then $\beta$ corresponds to $=\mathfrak P_{\rm 1}^{e_{\rm 1}}\mathfrak P_{\rm 2}^{e-e_{\rm 1}}$.

\subsection{$\ell$-isogeny graph}
It is obvious that an ordinary curve and a supersingular curve can never be isogenous. So we can consider the ordinary case and the supersingular case respectively. Given a prime number $\ell$, $\ell\neq p$, we introduce the isogeny graph in the following.

 \begin{itemize}
 \item
Supersingular case:  The $\ell$-isogeny graph $G_{\ell}^{s}$ of supersingular elliptic curves  over $\bar k$ where the vertices are the $\bar k$-isomorphism classes of elliptic curves  and the directed edges correspond to the $\ell$-isogenies, is a $\ell+1$ regular connected Ramanujan graph\cite{charles2009cry,pizer1990ram} . Every vertex has  $\ell+1$ out-degree. Consider the supersingular $\ell$-isogeny graph  $G_{\ell}^{s}(k,t)$ over $k$ with trace $t$, where the vertices are the $k$-isomorphism classes of elliptic curves and the edges correspond to the $l$-isogenies defined over $k$. If all endomorphisms are defined over $k$, we have $G_{\ell}^{s}(\mathbb F_{p^n},\pm2p^\frac{n}{2})\cong G_{\ell}^{s}$\cite{adj2019iso}. Otherwise, the $\ell$-isogeny graphs over $k$ are similar to the ordinary case. In fact, we might as well suppose that $k=\mathbb F_{p^2}$  since every supersingular $j$-invariant is defined over $\mathbb F_{p^2}$.
\item

Ordinary case: Let $G_{\ell}^{o}( k,t)$ be the $\ell$-isogeny graph of ordinary elliptic curves defined over $k$ with trace $t$. The structure of  $G_{\ell}^{o}( k,t)$ was clear by the following three propositions from Kohel's  thesis\cite{kohel1996endomorphism} and then Sutherland put forward the definition of $\ell$-volcano\cite{sutherland2013isogeny}. Suppose $\mathbb Z[\pi_{q,E}]$ has conductor $\mathit f_0$, then $G_{\ell}^{o}(k,t)$ consists of finite such $\ell$-volcanoes with $d=v_{\ell}(\mathit f_0)$ where $v_{\ell}$ is the $\ell$-adic valuation. More precisely, the number of all vertices with endomorphism ring $\mathcal O$ is the ring class number $h(\mathcal O)$\cite{schoff1987nonsingular,sutherland2013isogeny,waterhouse1969abe}.
Any two volcanoes can't be connected, otherwise they are the same one. After enumerating the number of vertices in $G_{\ell}^{o}(k,t)$, $G_{\ell}^{o}( k,t)$  have
   $\sum\limits_{\substack{\mathbb{Z}[\pi]\subseteq \mathcal O\subseteq \mathcal O_K,\\ \mathcal O\ \rm {is\ maximal\ at}\ \ell }}\frac{h(\mathcal O)}{ord_{cl(\mathcal O)}(I)}$ connected components where $cl(\mathcal O)$ is the ring class group and $I$ is a prime ideal of $\mathcal O$ over $\ell$.

\end{itemize}

\begin{Definition}
 An $\ell$-volcano $V$ of level $d$ is a connected undirected graph whose vertices are partitioned into $V_0,\dots ,V_d$ such that the followings hold:
 \begin{enumerate}[1)]
   \item The subgraph on $V_0$ (the surface) is a regular graph of degree at most 2.
   \item For $i > 0$, each vertex in $V_i$ has exactly one neighbor in level $V_{i-1}$, and this accounts for every edge not on the surface.
   \item For $i < d$, each vertex in $V_i$ has degree $\ell+1$ and each vertex in $V_d$(the floor) has degree $1$.
 \end{enumerate}
\end{Definition}

\begin{pro}
Let $E_{\rm 2}/k$ be ordinary, $\beta: E_{\rm 2}\to E_{\rm 1}$ is an isogeny of prime degree $\ell$ prime to $p$, then the following three conditions will happen:
\begin{enumerate}[1)]
  \item $\rm{End}(\it{E}_{\rm{1}})\cong \rm{End}(\it{E}_{\rm{2}})$, in which case we say that $\beta$ is horizontal.
  \item $[\rm{End}(\it{E}_{\rm{2}}):\rm{End}(\it{E}_{\rm{1}})]=$ $\ell$, in which case we say that $\beta$ is descending.
  \item $[\rm{End}(\it{E}_{\rm{1}}):\rm{End}(\it{E}_{\rm{2}})]=$ $\ell$, in which case we say that $\beta$ is ascending.
\end{enumerate}

\end{pro}
\begin{proof}
 See \cite[Proposition 21]{kohel1996endomorphism} or \cite[Section 2.7]{sutherland2013isogeny}.
\end{proof}

Note that for $[\rm{End}(\it{E}_{\rm{1}}):\rm{End}(\it{E}_{\rm{2}})]=$ $\ell$, we may write $[\rm{End}(\it{E}_{\rm{2}}):\rm{End}(\it{E}_{\rm{1}})]$$ =\frac{1}{\ell}$ for the same meaning in the following.

\begin{pro}
 Let $\beta: E_{\rm 2}\to E_{\rm 1}$ be an isogeny of ordinary elliptic curves over $k$, then the following conditions are equivalent:
 \begin{enumerate}[1)]
   \item $\rm{End}(\it{E}_{\rm{1}})$ and $\rm{End}(\it{E}_{\rm{2}})$ are isomorphic.
   \item The left ideal $\rm{Hom}(\it{E}_{\rm{1}},\it{E}_{\rm{2}})\beta$ is an invertible ideal of $\rm{End}(\it{E}_{\rm{2}})$ with norm equal to deg$(\beta)$.
   \item There exists an isogeny $\psi: E_{\rm 2}\to E_{\rm 1}$ of degree relatively prime to deg$(\beta)$.
 \end{enumerate}

\end{pro}
\begin{proof}
  See \cite[Proposition 22]{kohel1996endomorphism} or \cite[Section 2.9]{sutherland2013isogeny}.
\end{proof}

\begin{pro}
Let E/k be an ordinary elliptic curve with endomorphism ring $\mathcal O$ of discriminant $D$, let $\ell$ be a prime different from p, and let $\left(\frac{\ast}{\ast}\right)$ be the Kronecker symbol. The following isogenies are defined over $k$.
\begin{enumerate}[1)]
  \item If $\mathcal O$ is maximal at $\ell$, then there are $1+\left(\frac{D}{\ell}\right)$ isogenies of degree $\ell$ from $E$ to curves with endomorphism ring  $\mathcal O$.
  \item If $\mathcal O$ is not maximal at $\ell$, then there are no isogenies of degree $\ell$ from $E$ to curves with endomorphism ring $\mathcal O$.
  \item If there exist more than $1+\left(\frac{D}{\ell}\right)$ distinct isogenies from $E$ of degree $\ell$ then
all isogenies from $E$ of degree $\ell$ are defined over k, and up to $k$-isomorphism of $E'$, there are exactly
 \begin{equation*}
\left(l-\left(\frac{D}{\ell}\right)\right)[\mathcal O^{\ast}:\mathcal O'^{\ast}]^{-1}
\end{equation*}
elliptic curves $E'$ and $[\mathcal O^{\ast}:\mathcal O'^{\ast}]$ distinct isogenies $E \to E'$ of degree $\ell$ such that $[\rm{End}(\it{E}):\rm{End}(\it{E}')]=$ $\ell$ where $\rm{End}(\it{E}')=\mathcal O'$.
\end{enumerate}
\end{pro}
\begin{proof}
 See \cite[Proposition 23]{kohel1996endomorphism} or \cite[Section 2.10]{sutherland2013isogeny}.
\end{proof}

Note that $\vert \mathcal O^{\ast} \vert=2$ if $j_E\ne 0, 1728$. If $ j_E=1728$, then   $\rm{End}(\it{E})=\mathbb Z[i]$ and if $j_E=0$, then $\rm{End}(\it{E})=\mathbb Z[\frac{\rm 1+\sqrt{-3}}{2}]$ .

 \subsection{Ideals of imaginary quadratic orders}

This subsection can be found in \cite{buchmann2007binary,cox2011primes}. Suppose now $K$ is an imaginary quadratic field,  $\mathcal O$ is an order of it. An $\mathcal O$-ideal $I$ is an additive subgroup of $\mathcal O$ which is an $\mathcal O$-module with respect to multiplication. $I$ is called primitive if  it cannot be written as $tJ$ for some  $t\in \mathbb Z$ and $J$ is an $\mathcal O$-ideal. $I$ is called proper if $\{ \alpha \in K \vert \alpha I\subseteq I\}=\mathcal O$. $I$ is invertible if and only if it is proper. Since $\rm{rank}_\mathbb Z(\mathcal O)=2$, its ideals also have a basis consisting of two elements. Suppose $K$ has discriminant $D_0$, and let $\gamma$ be $\pm \frac {\sqrt{D_0}}{2} $ if $D_0\equiv0\ (\rm{mod}\ 4)$ or $\frac {1\pm \sqrt{D_0}}{2}$ if $D_0\equiv1\ (\rm{mod}\ 4)$. Suppose $\mathcal O$ has conductor $\mathit f$, then the discriminant of $\mathcal O$ is $\mathit f^2D_0$. Any primitive $\mathcal O$-ideal $I$ can be written as $\mathbb Za+\mathbb Z(b+\mathit f \gamma)$\cite[Proposition 8.4.5]{buchmann2007binary}. For the general integral ideals, they have the form $\mathbb Zat+\mathbb Zt(b+\mathit f \gamma)$ where $t,a \in \mathbb Z$. The norm of $I$ in  $\mathcal O$ is $[\mathcal O:I]=t^{\rm 2}a$.

\section{The index of $\rm{Hom}_{\it{k}}(\it{E}_{\rm{1}},\it{E}_{\rm{2}})\beta$ in $\rm{End}_{\it{k}}(\it{E}_{\rm{2}})$}

\begin{pro}
  Given two elliptic curves $E_{\rm 1}$, $E_{\rm 2}$ defined over a finite field $k$ with the same trace, if $\rm{End}_{\it{k}}(\it{E}_{\rm{2}})=\rm{End}(\it{E}_{\rm{2}})$, then $\rm{Hom}_{\it{k}}(\it{E}_{\rm{1}},\it{E}_{\rm{2}})=\rm{Hom}(\it{E}_{\rm{1}},\it{E}_{\rm{2}})$.
\end{pro}
\begin{proof}
  If $\rm{Hom}_{\it{k}}(\it{E}_{\rm{1}},\it{E}_{\rm{2}})\ne \rm 0$, then rank$_{\mathbb Z} \rm{Hom}_{\it{k}}(\it{E}_{\rm{1}},\it{E}_{\rm{2}})=\rm{rank}_{\mathbb Z} \rm{End}_{\it{k}}(\it{E}_{\rm{2}})$, since $\rm{Hom}_{\it{k}}(\it{E}_{\rm{1}},\it{E}_{\rm{2}})\beta$ is a non-zero ideal of $\rm{End}_{\it{k}}(\it{E}_{\rm{2}})$ for any non-zero $\beta \in \rm{Hom}_{\it{k}}(\it{E}_{\rm{2}},\it{E}_{\rm{1}})$.
  Since $\rm{Hom}_{\it{k}}(\it{E}_{\rm{1}},\it{E}_{\rm{2}})\subseteq \rm{Hom}(\it{E}_{\rm{1}},\it{E}_{\rm{2}})$ with the same rank as $\mathbb Z$-module, for any $\alpha \in \rm{Hom}(\it{E}_{\rm{1}},\it{E}_{\rm{2}})$, there exists $m\in \mathbb Z$ such that $m\alpha \in \rm{Hom}_{\it{k}}(\it{E}_{\rm{1}},\it{E}_{\rm{2}})$. If $\alpha$ is separable, since $\rm{ker}(\it{m})\subseteq \rm{ker} (\it{m}\alpha)$, there exists $\alpha' \in \rm{Hom}_{\it{k}}(\it{E}_{\rm{1}},\it{E}_{\rm{2}})$ such that $m\alpha=\alpha' m$  by \cite[Corollarly \uppercase\expandafter{\romannumeral3} 4.11]{silverman2009arithmetic}, hence $\alpha=\alpha'$. If $\alpha$ is not separable with inseparable degree $q'$, then $\alpha$ can be factored as $\alpha=\lambda \phi_{q',E_{\rm 1}}$ where $\lambda$ is separable from $E_{\rm 1}^{(q')}$ to $E_{\rm 2}$.
$\lambda$ can be proved to be defined over $k$ since $E_{\rm 1}^{(q')}$ is also defined over $k$, hence $\alpha$ is defined over $k$.
\end{proof}

Hence for the elliptic curves over $k$ whose endomorphisms are all defined over $k$, it is enough to consider $\rm{Hom}(\it{E}_{\rm{1}},\it{E}_{\rm{2}})$. In fact, with the help of the study of Waterhouse and Kohel, it is relatively simple to get the index of  $\rm{Hom}(\it{E}_{\rm{1}},\it{E}_{\rm{2}})\beta$ in $\rm{End}(\it E_{\rm2})$ when $E_2$ is supersingular or $E_2$ is ordinary and $\rm{Hom}(\it{E}_{\rm{1}},\it{E}_{\rm{2}})\beta$ is an invertible ideal of  $\rm{End}(\it E_{\rm 2})$. But it is not trivial to get the answer for any $\beta$ when $E_2$ is ordinary and $\rm{Hom}(\it{E}_{\rm{1}},\it{E}_{\rm{2}})\beta$ is not invertible.

\subsection{Supersingular case}

\begin{thm}

Given two supersingular elliptic curves $E_{\rm 1}$, $E_{\rm 2}$ defined over $k$, and an isogeny $\beta: E_2 \to E_1$, then
 \begin{equation*}
   [\rm{End}(\it{E}_{\rm{2}}): \rm{Hom}(\it{E}_{\rm1}, E_{\rm2})\beta]=[\rm{End}(\it{E}_{\rm{1}}): \beta\rm{Hom}(\it{E}_{\rm1}, E_{\rm2})]=(\rm{deg}\beta)^2
 \end{equation*} and $\rm{Hom}(\it{E}_{\rm1}, E_{\rm2})\beta$ is the kernel ideal corresponding to $\beta$.
\end{thm}
\begin{proof}
For every left(right) ideal $I$ of a maximal order of $B_{p,\infty}$, the reduced norm of $I$ can be defined as
\begin{equation*}
\rm{Nrd}(\it I)=\rm{gcd}\{\rm{Nrd}(\alpha):\alpha \in \it I\}.
\end{equation*}
Choose a prime $\ell$ prime to $\rm{deg}\beta$. Since the $\ell$-isogeny graph is connected, there is an isogeny $\alpha_0: E_{\rm 1} \to E_{\rm 2}$ with degree $\ell$-power. Hence
\begin{equation*}
\rm{Nrd}(Hom(\it{E}_{\rm 1},\it{E}_{\rm 2})\beta)=\rm{gcd}\{\rm{deg}(\alpha) \rm{deg}(\beta): \alpha \in Hom(\it{E}_{\rm 1},\it{E}_{\rm 2})\}= \rm{deg}(\beta).
\end{equation*}
Since the index of an left(right) ideal in a maximal order of the quaternion algebra is the square of the reduced norm of the ideal, we have \begin{equation*}[\rm{End}(\it{E}_{\rm{2}}): \rm{Hom}(\it{E}_{\rm1}, E_{\rm2})\beta]=[\rm{End}(\it{E}_{\rm{1}}): \beta\rm{Hom}(\it{E}_{\rm1}, E_{\rm2})]=(\rm{deg}\beta)^2.\end{equation*}
 Since $\rm{H}(\beta) \subseteq H(\rm{Hom}(\it{E}_{\rm{1}},\it{E}_{\rm{2}})\beta) \subseteq \rm{H}([deg(\beta)])$ and $\rm{H}(\alpha_{\rm 0} \beta)\cap  \it H([\rm{deg}\beta])=\rm{H}(\beta)$, it follows that
 \begin{equation*}
   \rm H(\rm{Hom}(\it{E}_{\rm{1}},\it{E}_{\rm{2}})\beta)=\rm{H}(\beta).
\end{equation*}
 Hence $\rm{Hom}(\it{E}_{\rm1}, E_{\rm2})\beta$ is the kernel ideal of $\beta$.
\end{proof}

\begin{remark} Kohel had proved that all ideals of $\rm{End}(\it{E}_{\rm{2}})$ could be in form of $\rm{Hom}(\it{E}_{\rm{1}},\it{E}_{\rm{2}})\beta$ in his thesis\cite{kohel1996endomorphism}. In fact, there is a one-to-one correspondence between the  isogenies from $E_{\rm 2}$ and the left ideals of $\rm{End}(\it{E}_{\rm{2}})$. Under this correspondence, the reduced norm of the ideal is  equal to the degree of the corresponding isogeny.
\end{remark}


 If $\rm{Hom}_{\it{k}}(\it{E}_{\rm1}, E_{\rm2})\ne \rm 0$ and $\rm{End}_{\it{k}}(\it{E}_{\rm{2}})(\ne \rm{End}(\it{E}_{\rm{2}}))$ is isomorphic to an imaginary quadratic order, we can analyze it similarly to the ordinary case which is omitted here.
\subsection{Ordinary case}
In the following, we say an isogeny $\beta$  doesn't have any backtracking if $\beta$  can't be written as $m\beta'$ for some $m\in \mathbb Z$ and  some isogeny $\beta'$. We assume $E_1$ and $ E_2$ are ordinary, and let $K$ be the endomorphism algebra of $E_2$, and denote its algebraic integer ring $\mathcal O_K$ by $\mathbb Z +\mathbb Z \gamma$. We begin with the following lemma and its corollary which guarantee the assumptions about the relations of the endomorphism rings in Lemma 3.3 and Theorem 3.2. Next we get the results for the isogenies of prime powers and finally use the results of the prime power case to prove the general case by the recurrence method.

\begin{lem}
 Given two ordinary elliptic curves $E_{\rm 1}, E_{\rm 2}$ defined over $k$. If there is an isogeny between them of degree $\ell^e$ where $\ell \ne p$, then $[\rm{End}(\it{E}_{\rm{2}}):\rm{End}(\it{E}_{\rm{1}})]=\ell^{e'}$ where $\vert e' \vert \leqslant e$ . Conversely, if $[\rm{End}(\it{E}_{\rm{2}}):\rm{End}(\it{E}_{\rm{1}})]=\ell^{e'}$, then the degree of every isogeny between them is divided by $\ell^{\vert e'\vert}$.
\end{lem}

\begin{proof}
  Since distinct $\ell$-volcanoes are unconnected, $E_{\rm 1}$ and $E_{\rm 2}$ must be in the same $\ell$-volcano. Let $\beta: E_2\to E_1$ be an isogeny of degree $\ell^e$, $\beta$ can be looked as a path walking in the $\ell$-volcano since all $\ell$-isogenies are represented as a edge in the $\ell$-volcanos. We may as well assume that $\beta$ does not have backtracking, then the path can be showed as
\begin{equation*}
E'_0:=E_{\rm 2}\stackrel{\beta_{\rm 1}}{\to}E'_{\rm 1}\stackrel{\beta_{\rm 2}}{\to}\cdots \stackrel{\beta_e}{\to}E'_e=:E_{\rm 1}
\end{equation*}
where $\rm{ker}\beta_{\it i}=$$\ell^{e-i}\beta_{i-1}\cdots \beta_{\rm 1} ({\rm{ker}}\beta)$ and $\beta_i$ are also defined over $k$ for all $i$. According to Proposition 2.3, $[\rm{End}(\it{E}'_{i}):\rm{End}(\it{E}'_{i+\rm1})]=\rm 0$, $\ell$ or $\ell^{-1}$ for all $i$. Hence  $[\rm{End}(\it{E}_{\rm{2}}):\rm{End}(\it{E}_{\rm{1}})]=\ell^{e'}$ where $\vert e' \vert \leqslant e$.

Conversely, suppose $e'>0$, then there is an elliptic curve $E'_{\rm 2}$ with the same endomorphism ring with $E_{\rm 2}$ descending directly to $E_{\rm 1}$  such that every isogeny $\beta$ between them must pass  $E'_{\rm 2}$. Hence $\ell^{e'}\mid \rm{deg}\beta$. If $e'<0$, consider $\widehat{\beta}$ and the lemma holds. One can also refer to \cite[Proposition 5]{kohel1996endomorphism} for another proof.
\end{proof}

\begin{cor}
  Given two ordinary elliptic curves $E_{\rm 1}, E_{\rm 2}$ defined over $k$. If there is an isogeny between them of degree $m$ with factorization $p^e\ell_{\rm 1}^{e_{\rm 1}}\cdots \ell_s^{e_s}$, then $[\rm{End}(\it{E}_{\rm{2}}):\rm{End}(\it{E}_{\rm{1}})]$ has the form $\ell_{\rm 1}^{e'_{\rm 1}}\cdots \ell_s^{e'_s}$ where $\vert e'_i \vert \leqslant e_i$ for all $i$. Conversely, if $[\rm{End}(\it{E}_{\rm{2}}):\rm{End}(\it{E}_{\rm{1}})]$ has the form $\ell_{\rm 1}^{e'_{\rm 1}}\cdots \ell_s^{e'_s}$, then the degree of every isogeny between them is divided by $\ell_{\rm 1}^{\vert e'_{\rm 1}\vert }\cdots \ell_s^{\vert e'_s \vert}$.
\end{cor}
\begin{proof}
 Suppose the isogeny is $\beta: E_2\to E_1$. Since $\rm{ker}\beta$ is a finite abelian group and can be a direct sum of subgroups of different prime powers. Let $\rm{ker}\beta=\it{G}_{\rm 1}+\cdots+G_s+G_{s+\rm 1}$ where $\vert G_i\vert=\ell_i^{e_i}$ for $i<s+1$ and $\vert G_{s+1}\vert \mid p^e$.
 Let
 \begin{equation*}
E_{\rm 2}=:E'_0\stackrel{\beta_{\rm 1}}{\to}E'_{\rm 1}\stackrel{\beta_{\rm 2}}{\to}\cdots \stackrel{\beta_s}{\to}E'_{s}\stackrel{\beta_{s+1}}{\to}E'_{s+1}:=E_{\rm 1}.
\end{equation*}
where $\rm{ker}\beta_{\it{i}}=\beta_{\it{i}-\rm1}\cdots \beta_{\rm 1} \it{G}_i$ and they are both defined over $k$. For every $i<s+1$, by Lemma 3.1, $[\rm{End}(\it{E}'_{i-\rm1}):\rm{End}(\it{E}'_{i})]=\ell_{i}^{e'_i}$ where $\vert e'_i \vert \leqslant e_i$. Hence $[\rm{End}(\it{E}_{\rm{2}}):\rm{End}(\it{E}_{\rm{1}})]$ has the form $\ell_{\rm 1}^{e'_{\rm 1}}\cdots \ell_s^{e'_s}$ where $\vert e'_i \vert \leqslant e_i$ for all $i$, since $[\rm{End}(\it{E}'_{s}):\rm{End}(\it{E}'_{s+\rm 1})]=\rm 0$.

 Conversely, see \cite[Proposition 5]{kohel1996endomorphism} or by Lemma 3.1.
\end{proof}

\begin{lem}
  Let $\beta: E_2 \to E_1$ be a separable isogeny of elliptic curves over $k$, then we have $\rm{I}({\rm{ker}}\beta)=\rm{Hom}(\it{E}_{\rm 1}, E_{\rm 2})\beta$.
\end{lem}
\begin{proof}
  For any $\alpha \in \rm{I}({\rm{ker}}\beta)$, since $\rm{ker}\beta \subseteq \rm{ker}\alpha$ and $\beta$ is separable, there exists $\lambda \in \rm{Hom}(\it{E}_{\rm 1}, E_{\rm 2})$ such that $\alpha= \lambda\beta$. Hence $\rm{I}({\rm{ker}}\beta)\subseteq \rm{Hom}(\it{E}_{\rm 1}, E_{\rm 2})\beta$.
  So $\rm{I}({\rm{ker}}\beta)=\rm{Hom}(\it{E}_{\rm 1}, E_{\rm 2})\beta$, since $\rm{Hom}(\it{E}_{\rm 1}, E_{\rm 2})\beta \subseteq \rm{I}({\rm{ker}}\beta)$ according to definition.
\end{proof}

\begin{lem}
Given two ordinary elliptic curves $E_{\rm 1}, E_{\rm 2}$ defined over $k$ and $\beta:  E_{\rm 2}\to E_{\rm 1}$ an isogeny of degree $\ell^e$($\ell$ can be $p$), and suppose $[\rm{End}(\it{E}_{\rm{2}}):\rm{End}(\it{E}_{\rm{1}})]$ $=\ell^{e'}$ where $\vert e' \vert \leqslant e$, then $[\rm{End}(\it{E}_{\rm{2}}):\rm{Hom}(\it{E}_{\rm{1}},\it{E}_{\rm{2}})\beta]$ $=\ell^{e'}\rm{deg} \beta $ if $e'>0$  or $\rm{deg} \beta$, otherwise. In addition, suppose $\rm{End}(\it{E}_{\rm{2}})$ has conductor $\mathit f$, writing $\rm{End}(\it{E}_{\rm{2}})=\mathbb Z+\mathbb Z\mathit f\gamma$, and $\beta=\ell^{e_0}\beta'$ where $\beta'$ has no backtracking, then $\rm{Hom}(\it{E}_{\rm{1}},\it{E}_{\rm{2}})\beta=\mathbb Z \rm{deg}\beta +\mathbb Z$$ \ell^{e_0+e'}(b+\mathit f\gamma)$ if  $e'>0$ and $ \rm{Hom}(\it{E}_{\rm{1}},\it{E}_{\rm{2}})\beta=\mathbb Z \rm{deg}\beta$$  +\mathbb Z\ell^{e_0}(\ell^{\vert e'\vert}b+\mathit f\gamma)$ where $\ell^{\vert e' \vert} \vert \mathit f$, otherwise.
\end{lem}

\begin{proof}
First, we assume $\beta$ has no backtracking.  If $\ell\ne p$, then $\beta$ is separable, hence by Lemma 3.2, $\rm{Hom}(\it{E}_{\rm{1}},\it{E}_{\rm{2}})\beta=\rm I({\rm{ker}}\beta)$. It is enough to prove for $e'\geq 0$ . For $e'<0$, we can consider the dual isogeny $\widehat \beta$ where $\rm{I}({ \rm{ker}} \widehat\beta)=\overline{\rm{I}({ \rm{ker}}\beta)}$ (note that $\bar \gamma=-\gamma$ or $-\gamma+1$, so that $\overline{ \rm I({ \rm{ker}}\beta)}$ can written like above).

Now assume $e'>0$. Consider $E_2$ and $E_1$ in the $\ell$-volcano, then there is an elliptic curve $E'_{\rm 2}$ with the same endomorphism ring with $E_{\rm 2}$ descending directly to $E_{\rm 1}$  such that the path $\beta$ must pass  $E'_{\rm 2}$. Write $\beta=\beta_{\rm 2}\beta_{\rm 1}$ where $\beta_{\rm 1}: E_{\rm 2} \to E'_{\rm 2}$ and $\beta_{\rm 2}$ is the descending directly path from $E'_{\rm 2}$ to $E_{\rm 1}$.  If $e'=0$, let $E'_{\rm 2}=E_{\rm 1}$ and $\beta_{\rm 2}=id$.

Since $\beta$ has no backtracking, $\rm{ker}\beta$ is a cyclic group of order $\ell^e$. Since $\ell^e$ is the smallest integer belonging to $\rm{I}({\rm{ker}}\beta)$, we have $\rm{I}({ \rm{ker}}\beta)\cap \mathbb Z=\mathbb Z$$ \ell^e$.  Let $\rm{ker}\beta=\langle \it P\rangle$. Similarly, $\rm{ker}\beta_{\rm 1}$ is a cyclic group of order $\ell^{e-e'}$, and $\ell^{e'}\beta_{\rm 1} P=\infty$, then $\rm{ker}\beta_{\rm 1}=$ $\langle \ell^{e'}P\rangle$.

Since $\rm{End}(\it{E}_{\rm{2}})\cong \rm{End}(\it{E}'_{\rm 2})$, $\rm{I}({\rm{ker}}\beta_{\rm 1})$ is an invertible ideal of norm equal to deg $\beta_{\rm 1}=\ell^{e-e'}$  by Proposition 2.4. Similarly, it can be proved that $\rm{I}({\rm{ker}}\beta_{\rm 1})\cap\mathbb Z$ $=\mathbb Z \ell^{e-e'}$,  hence $\rm{I}({\rm{ker}}\beta_{\rm 1})$ has the form $\mathbb Z \ell^{e-e'} +\mathbb Z (b+\mathit f\gamma)$ for some $b\in \mathbb Z$.  Since $\rm{I}({\rm{ker}}\beta)$ is also an ideal of $\rm{End}(\it{E}_{\rm{1}})$, $\rm{I}({\rm{ker}}\beta) $ is contained in $\mathbb Z+\mathbb Z\ell^{e'}\mathit f\gamma$. Since $\ell^{e'}(b+\mathit f\gamma )(P)=(b+\mathit f\gamma)(\ell^{e'}P)=\infty$, it follows that $\ell^{e'}(b+\mathit f\gamma )\in$$ \rm{I}({\rm{ker}}\beta)$. Hence $\rm{I}({\rm{ker}}\beta)=$ $\mathbb Z \ell^e +\mathbb Z \ell^{e'}(b+\mathit f\gamma)$ and  $[\rm{End}(\it{E}_{\rm{2}}):\rm I({\rm{ker}}\beta)]= $ $\ell^{e'}\rm{deg} \beta $.

If $\ell=p$, then $\rm{End}(\it{E}_{\rm{1}})\cong \rm{End}(\it{E}_{\rm{2}})$ and  $\rm{Hom}(\it{E}_{\rm{1}},\it{E}_{\rm{2}})\beta$ is an invertible ideal of norm $\rm{deg}\beta$ by Proposition 2.4, since $\mathbb Z[\pi_{q,E_{\rm 2}}]$ is maximal at $p$.

 If  $\beta$ has backtracking, since $\rm{Hom}(\it{E}_{\rm{1}},\it{E}_{\rm{2}})$$\ell^{e_0}\beta'= $ $\ell^{e_0}\rm{Hom}(\it{E}_{\rm{1}},\it{E}_{\rm{2}})\beta'$, the lemma holds.
\end{proof}

\begin{cor}
Given ordinary elliptic curves $E_{\rm 1}, E_{\rm 2}$ defined over $k$ where $[\rm{End}(\it{E}_{\rm{2}}): \rm{End}(\it{E}_{\rm{1}})]=$ $\ell^{e'}$, $\ell\ne p$, $e'\in \mathbb Z_{>0}$. Let $\beta:E_2\to E_1$ be an isogeny without backtracking of prime power degree, and suppose $\rm{End}(\it{E}_{\rm{2}})=\mathbb Z+\mathbb Z\mathit f\gamma$, then $\rm{Hom}(\it{E}_{\rm{1}},\it{E}_{\rm{2}})=$ $\mathbb Z\widehat\beta+\mathbb Z \frac{\ell^{e'}(b+\mathit f\gamma)\widehat \beta}{\rm{deg} \beta}$ for some integer $b$. In addition, if $\rm{deg} \beta=\ell^{e'}$, then $\rm{Hom}(\it{E}_{\rm{1}},\it{E}_{\rm{2}})=\mathbb Z\widehat\beta+\mathbb Z\mathit f\gamma\widehat\beta$.

\end{cor}
\begin{proof}
Since $[\rm{End}(\it{E}_{\rm{2}}): \rm{End}(\it{E}_{\rm{1}})]=$ $\ell^{e'}$, we have $\ell\mid \rm{deg}\beta$. By Lemma 3.3,
\begin{equation*}\rm{Hom}(\it{E}_{\rm{1}},\it{E}_{\rm{2}})\widehat\beta=\mathbb Z \rm{deg}\beta +\mathbb Z \ell^{e'}(b+\mathit f\gamma)
\end{equation*} for some integer $b$, then $\rm{Hom}(\it{E}_{\rm{1}},\it{E}_{\rm{2}})=\frac{\rm{Hom}(\it{E}_{\rm{1}},\it{E}_{\rm{2}})\beta\widehat\beta}{\rm{deg}\beta}$ $=\mathbb Z\widehat\beta+\mathbb Z \frac{\ell^{e'}(b+\mathit f\gamma)\widehat \beta}{\rm{deg} \beta}$. If  $\rm{deg} \beta=\ell^{e'}$, then $b$ can be 0.
\end{proof}

\begin{cor}
  Given ordinary elliptic curves $E_{\rm 1}, E_{\rm 2}$ defined over $k$  where $[\rm{End}(\it{E}_{\rm{2}}):\rm{End}(\it{E}_{\rm{1}})]=\ell^{e'}$ where $e'\in \mathbb Z$. Then every isogeny from $E_{\rm 2}$ to $E_{\rm 1}$ can correspond to an ideal of $\rm{End}(\it{E}_{\rm{2}})$ if and only if  $e' \leqslant 0$.
\end{cor}
\begin{proof}
If $e'=0$, then $\rm{End}(\it{E}_{\rm{1}})\cong\rm{End}(\it{E}_{\rm{2}})$. Hence for any isogeny $\beta\in \rm{Hom}(\it{E}_{\rm{2}},\it{E}_{\rm{1}})$, $\beta$ corresponds to the invertible ideal $\rm{Hom}(\it{E}_{\rm{1}},\it{E}_{\rm{2}})\beta$. Otherwise, $\ell\ne p$ and $\ell \mid \rm{deg}(\beta)$. Factor $\beta$ as $\beta_2\beta_1$ where deg$\beta_2$ is $\ell$-power and $\ell \nmid$ deg$\beta_1$. Let $E_3$ be the target elliptic curve of $\beta_1$ where $\rm{End}(\it{E}_{\rm{3}})\cong\rm{End}(\it{E}_{\rm{2}})$.

If $e'< 0$, since
\begin{equation*}
 \rm{ker}(\beta_2)\subseteq H(\rm{Hom}(\it{E}_{\rm{1}},\it{E}_{\rm{3}})\beta_{\rm2})\subseteq E_{\rm 3}[\rm{deg}\beta_2]
\end{equation*}
and $E_{\rm 3}[\ell]\nsubseteq \rm{H}(\rm{Hom}(\it{E}_{\rm{1}},\it{E}_{\rm{3}})\beta_{\rm2})$ by Lemma 3.3 (otherwise $\frac{\mathit f\gamma}{\ell}  \in \rm{End}(\it{E}_{\rm{3}})$, which is a contradiction), we obtain $\rm{H}(\rm{Hom}(\it{E}_{\rm{1}},\it{E}_{\rm{3}})\beta_{\rm2})=\rm{ker}(\beta_2)$, hence $\rm{Hom}(\it{E}_{\rm{1}},\it{E}_{\rm{3}})\beta_{\rm2}$ is the kernel ideal for $\beta_2$. So $\beta$ can correspond to an ideal by Proposition 2.2, since $\beta_1$ can also correspond to an ideal.

 If $e'> 0$, suppose that $\beta$ has no backtracking. If $\beta$ can correspond to some ideal, then $\beta\widehat{\beta_1}=\beta_2\rm{deg}\beta_1$ corresponds to an ideal. Hence $\beta_2$ corresponds to an ideal. By lemma 3.3, we have $E_{\rm 3}[\ell^{e'}]\subseteq \rm{H}(\rm{Hom}(\it{E}_{\rm{1}},\it{E}_{\rm{3}})\beta_{\rm2})$, then the isogeny corresponding to the kernel ideal $\rm{Hom}(\it{E}_{\rm{1}},\it{E}_{\rm{3}})\beta_{\rm2}$ can be written as $\ell^{e'}\beta'$ for some isogeny $\beta'$ from $E_3$, hence $\rm{Hom}(\it{E}_{\rm{1}},\it{E}_{\rm{3}})\beta_{\rm2}$ isn't the kernel ideal for $\beta_2$ since $\beta_2$ has no backtracking. Let $J$ be the kernel ideal of $\beta_2$ such that $\rm{H}(\it J)=\rm{H}(\beta_{\rm2})$, then $J=\rm{I}(H(\it J))=\rm{I}(ker(\beta_2))=\rm{Hom}(\it{E}_{\rm{1}},\it{E}_{\rm{3}})\beta_{\rm2}$ since $\beta_2$ is separable. Thus we obtain a contradiction. $\beta$ can not correspond to an ideal.

\end{proof}

Now take out the condition that the degrees of the isogenies are prime powers. For simplicity, we define $\rho(e)$ as
\begin{equation*}
\rho(e)=\left \{
\begin{aligned}
  e, & \   \rm{if}\  \it e>\rm0. \\
  0, & \   \rm{otherwise}.
\end{aligned}
\right.
\end{equation*}

\begin{thm}
Given two ordinary elliptic curves $E_{\rm 1}, E_{\rm 2}$ defined over $k$ and $\beta:  E_{\rm 2}\to E_{\rm 1}$ an isogeny of degree $m$, let the prime factorization of $m$ be $p^e\ell_{\rm 1}^{e_{\rm 1}}\cdots \ell_s^{e_s}$ where $e\in \mathbb Z_{\geqslant 0}$ and $e_i \in \mathbb Z_{>0}$. Suppose $[\rm{End}(\it{E}_{\rm{2}}):\rm{End}(\it{E}_{\rm{1}})]=$ $\ell_{\rm 1}^{e'_{\rm 1}}\cdots \ell_s^{e'_s}$ where $\vert e'_i \vert \leqslant e_i$ for all $i$, then
\begin{equation*}
[\rm{End}(\it{E}_{\rm{2}}):\rm{Hom}(\it{E}_{\rm{1}},\it{E}_{\rm{2}})\beta]=\prod\limits_{\it{i}=\rm 1}^s \ell_i^ {\rho(e'_i)}\rm{deg} \beta.
\end{equation*} In addition, suppose $\rm{End}(\it{E}_{\rm{2}})$ has conductor $\mathit f$, writing $\rm{End}(\it{E}_{\rm{2}})=\mathbb Z+\mathbb Z\mathit f\gamma$, and $\beta=m'\beta'$ where $m'\in \mathbb Z$  and $\beta'$ has no backtracking, then
\begin{equation*}
\rm{Hom}(\it{E}_{\rm{1}},\it{E}_{\rm{2}})\beta=\it{m}'(\mathbb Z \rm{deg}\beta' +\mathbb Z \prod\limits_{\it{i}=\rm 1}^s \ell_{\it i}^ {\rho(\it e'_i)}(\it b\prod\limits_{\it{i}=\rm 1}^s \ell_{\it i}^ {\rho(\it e'_i)-e'_i}+\mathit f\gamma))
\end{equation*} for some integer $b$ where $\prod\limits_{i=1}^s \ell_i^ {\rho(e'_i)-e'_i}\mid \mathit f$.
\end{thm}

\begin{proof}
Similarly, it is enough to prove the case $\beta$ has no backtracking. If $\beta$ has inseparable degree $p^e$ and $e>0$, then we factor $\beta$ as
\begin{equation*}
E_{\rm 2}\stackrel{\phi_{p^e,E_{\rm 2}}}{\longrightarrow}E_{\rm 2}^{(p^e)}\stackrel{\lambda}{\longrightarrow}E_{\rm 1}
\end{equation*}
where $\lambda$ is separable of degree $\ell_{\rm 1}^{e_{\rm 1}}\cdots \ell_s^{e_s}$. We consider $\widehat\beta=\widehat \phi_{p^e,E_{\rm 2}}\widehat\lambda $ and $\rm{Hom}(\it{E}_{\rm{2}},\it{E}_{\rm{1}})\widehat\beta $ instead.  
Assume $\beta$ is separable, let $\beta$ be factored as $\beta_{s+1} \cdots \beta_{\rm 2}\beta_{\rm 1}$ where $\rm{deg}\beta_{\it i}$ $=\ell_i^{e_i}$ and $\rm{deg}\beta_{\it s+\rm 1}=\it{p}^e$ as in the proof of Corollary 3.1. So we have
\begin{equation*}
E_{\rm 2}\stackrel{\beta_{\rm 1}}{\to}E'_{\rm 1}\stackrel{\beta_{\rm 2}}{\to}\cdots \stackrel{\beta_s}{\to}E'_{s}\stackrel{\beta_{s+1}}{\to}E'_{s+1}:=E_{\rm 1}.
\end{equation*}
 In the following, we prove the theorem recursively.

Suppose $\rm{End}(\it{E}_{\rm{2}})=\mathbb Z+\mathbb Z \mathit f \gamma$, let $\ell_{s+1}=p$ and $e_{s+1}=e$, then \begin{equation*}
\rm{End}(\it{E}'_i)=\mathbb Z+\mathbb Z \ell_{\rm 1}^{e'_{\rm 1}}\cdots \ell_i^{e'_i }\mathit f \gamma
                   \end{equation*}
for all $i$ and $e'_{s+1}=0$. We claim that
 $\rm{Hom}(\it{E}'_i,E_{\rm 2})\beta_i\cdots\beta_{\rm 1}$ has the form \begin{equation*}
\mathbb Z \prod\limits_{j=\rm 1}^i \ell_{\it j}^{\it e_j}+\mathbb  Z\prod\limits_{\it{j}=\rm 1}^i \ell_j^{\rho(e'_j)} (b\prod\limits_{j=\rm 1}^{i} \ell_j^ {\rho(e'_j)-e'_j}+\mathit f \gamma)
                                       \end{equation*}
for some integer $b$.

Since $\beta$ has no backtracking, every $\rm{ker} \beta_{\it{i}}$  and $\rm{ker} \beta$ are cyclic. Let $\rm{ker}\beta=\langle \it P\rangle$, then
\begin{equation*}
\rm{ker} (\beta_{\it{i}}\cdots\beta_{\rm 1})=\langle\prod\limits_{\it{j}=i+\rm 1}^{\it s+\rm 1} \ell_{\it j}^{\it e_j}\it P\rangle.
\end{equation*}Hence we have $\rm{I}({\rm{ker}} (\beta_{\it{i}}\cdots\beta_{\rm 1}))\cap \mathbb Z=\mathbb Z\prod\limits_{\it{j}=\rm 1}^{\it{i}}\ell_{\it j}^{\it e_j}$, since $\rm{ker}( \beta_{\it{i}}\cdots\beta_{\rm 1})$ is a cyclic group of order $\prod\limits_{j=\rm 1}^{i}\ell_{\it j}^{\it e_j}$. We also have $\rm{I}({\rm{ker}} (\beta_{\it{i}}\cdots\beta_{\rm 1}))\subseteq (\rm{End}(\it{E}_{\rm{2}})\cap \rm{End}(\it{E}'_i))= \mathbb Z+\mathbb Z \prod\limits_{j=\rm 1}^i \ell_j^{\rho(e'_j)}\mathit f \gamma$ where $\prod\limits_{j=\rm 1}^i \ell_j^{\rho(e'_j)-e'_j}\mid \it{f}$.

When $i=1$, by Lemma 3.3, $\rm{Hom}(\it{E}'_{\rm 1},E_{\rm 2})\beta_{\rm 1}$ has the form
\begin{equation*}
 \mathbb Z \rm{deg}\beta_{\rm 1}+\mathbb  Z \ell_{\rm 1}^{\rho(e'_{\rm 1})} (\ell_{\rm 1}^{\rho(e'_{\rm 1})-e'_{\rm 1}}\it b+\mathit f \gamma)
\end{equation*}
for some integer $b$. It holds.

If  $\rm{Hom}(\it{E}'_i,E_{\rm 2})\beta_i\cdots\beta_{\rm 1}$ has the form $\mathbb Z \prod\limits_{j=\rm 1}^i \ell_{\it j}^{\it e_j}+\mathbb  Z\prod\limits_{j=\rm 1}^i \ell_j^{\rho(e'_j)} (b\prod\limits_{j=\rm 1}^{i} \ell_j^ {\rho(e'_j)-e'_j}+\mathit f \gamma)$ for some integer $b$, we consider  $\rm{Hom}(\it{E}'_{\it i+\rm 1},E_{\rm 2})\beta_{\it i+\rm 1}\cdots\beta_{\rm 1}$. Since $\rm{I}({\rm{ker}}(\beta_{\it i+\rm 1}\cdots\beta_{\rm 1}))$  is an ideal of $\mathbb Z+\mathbb Z\mathit f \gamma$
 and $\rm{I}({\rm{ker}} (\beta_{\it i+\rm 1}\cdots\beta_{\rm 1}))\cap \mathbb Z=\mathbb Z\prod\limits_{\it{j}=\rm 1}^{\it i+\rm 1}\ell_{\it j}^{\it e_j}$, $\rm{I}({\rm{ker}}(\beta_{\it i+\rm 1}\cdots\beta_{\rm 1}))$ has the form
 $\mathbb Z \prod\limits_{j=\rm 1}^{\it i+\rm 1} \ell_{\it j}^{\it e_j}+\mathbb  Z t (b'+\mathit f \gamma)$
for some integers $t$ and $b'$, where $t\mid\prod\limits_{j=\rm 1}^{\it i+\rm 1} \ell_{\it j}^{\it e_j}$.
Since $\rm{I}({\rm{ker}} (\beta_{\it i+\rm 1}\cdots\beta_{\rm 1}))$$\subseteq  \mathbb Z+\mathbb Z \prod\limits_{j=\rm 1}^{\it i+\rm 1} \ell_j^{\rho(e'_j)}\mathit f \gamma$, we have  $\prod\limits_{j=\rm 1}^{\it i+\rm 1} \ell_j^{\rho(e'_j)} \mid t$. Let $t':=t/(\prod\limits_{j=\rm 1}^{\it i+\rm 1} \ell_j^{\rho(e'_j)})$, then $ \rm{I}({\rm{ker}}(\beta_{\it{i}+\rm{1}}\cdots\beta_{\rm 1}))$ has the form
\begin{equation*}
 \mathbb Z \prod\limits_{\it{j}=\rm 1}^{\it i+\rm 1} \ell_{\it j}^{\it e_j}+\mathbb  Z \it t' \prod\limits_{j=\rm 1}^{\it i+\rm 1} \ell_j^{\rho(e'_j)}(\it b'+\mathit f \gamma).
\end{equation*}
Since $\ell_{\it i+\rm 1}^{e_{\it i+\rm 1}} \prod\limits_{j=\rm 1}^i \ell_j^{\rho(e'_j)}$ $(\prod\limits_{j=\rm 1}^{i} \ell_j^ {\rho(e'_j)-e'_j}b+\mathit f \gamma)(\prod\limits_{j=i+2}^{s+1} \ell_j^{e_j}P)=\infty$, it follows that
\begin{equation*}
  \ell_{\it i+\rm 1}^{e_{\it i+\rm 1}} \prod\limits_{j=\rm 1}^i \ell_j^{\rho(e'_j)}(b\prod\limits_{j=\rm 1}^{i} \ell_j^ {\rho(e'_j)-e'_j}+\mathit f \gamma)\in \rm{I}({\rm{ker}}(\beta_{\it i+\rm 1}\cdots\beta_{\rm 1})),
\end{equation*} hence $t'\mid \ell_{\it i+\rm 1}^{e_{\it i+\rm 1}}$ and $\prod\limits_{j=\rm 1}^{i} \ell_j^ {\rho(e'_j)-e'_j}\mid b'$.

On the other hand, consider
\begin{equation*}
  \rm{Hom}(\it{E}_{\rm 2},E'_{\it i+\rm 1})(\widehat{\beta_{\it i+\rm 1}\cdots\beta_{\rm 1}})=\mathbb Z \prod\limits_{j=\rm 1}^{\it i+\rm 1} \ell_{\it j}^{\it e_j}+\mathbb  Z t' \prod\limits_{j=\rm 1}^{\it i+\rm 1} \ell_j^{\rho(e'_j)}(b'+\mathit f \bar{\gamma}).
\end{equation*}
 (note that $\widehat{\beta_{\it i+\rm 1}\cdots\beta_{\rm 1}}$ may be inseparable when $i=s$, so we don't use symbol like $\rm{I}({\rm{ker}}\beta)) here$.) Similarly, by Lemma 3.3, $\rm{Hom}(\it{E}'_i,E'_{\it i+\rm 1})\widehat{\beta_{\it i+\rm 1}}$ has the form
 \begin{equation*}
   \mathbb Z \ell_{\it i+\rm 1}^{e_{\it i+\rm 1}}+\mathbb  Z\ell_{\it i+\rm 1}^{\rho(e'_{\it i+\rm 1})}(\ell_{\it i+\rm 1}^{\rho(e'_{\it i+\rm 1})-e'_{\it i+\rm 1}}b''+\prod\limits_{j=\rm 1}^{i}\ell_j^{e'_j}\mathit f \bar{\gamma})
 \end{equation*} for some integer $b''$. Let $\rm{ker}(\widehat{\beta_{\it i+\rm 1}\cdots\beta_{\rm 1}})=\langle \it Q\rangle$, then $\rm{ker}(\widehat{\beta_{\it{i}}\cdots\beta_{\rm 1}})=\langle \widehat{\beta_{\it i+\rm 1}}(\it Q)\rangle$ is a cyclic group of order $\ell_{\rm 1}^{e_{\rm 1}}\cdots \ell_i^{e_i}$.
  Write $\ell_{\it i+\rm 1}^{\rho(e'_{\it i+\rm 1})}(\ell_{\it i+\rm 1}^{\rho(e'_{\it i+\rm 1})-e'_{\it i+\rm 1}}b''+\prod\limits_{j=\rm 1}^{i}\ell_j^{e'_j}\mathit f \bar{\gamma})=\alpha \widehat{\beta_{\it i+\rm 1}}$ for some $\alpha \in \rm{Hom}(\it{E}'_i,E'_{\it i+\rm 1})$, then we  have $\ell_{\rm 1}^{e_{\rm 1}}\cdots \ell_i^{e_i}\alpha \widehat{\beta_{\it i+\rm 1}}(Q)=\infty$. If $i\neq s$, then $\widehat{\beta_{\it i+\rm 1}\cdots\beta_{\rm 1}}$ is separable, and we have
 \begin{equation*}
   \ell_{\rm 1}^{e_{\rm 1}}\cdots \ell_i^{e_i}\ell_{\it i+\rm 1}^{\rho(e'_{\it i+\rm 1})}(\ell_{\it i+\rm 1}^{\rho(e'_{\it i+\rm 1})-e'_{\it i+\rm 1}}b''+\prod\limits_{j=\rm 1}^{i}\ell_j^{e'_j}\mathit f \bar{\gamma})\in \rm{Hom}(\it{E}_{\rm 2},E'_{\it i+\rm 1})(\widehat{\beta_{\it i+\rm 1}\cdots\beta_{\rm 1}}),
 \end{equation*}
 hence $t'\mid  \ell_{\rm 1}^{e_{\rm 1}}\cdots \ell_i^{e_i}$ and $\ell_{\it i+\rm 1}^{\rho(e'_{\it i+\rm 1})-e'_{\it i+\rm 1}}\mid b'$. Thus $t'=1$.
 If $i=s$, then $\rm{End}(\it{E}'_s)\cong \rm{End}(\it{E}'_{s+\rm 1})$ and there is an isogeny $\alpha'$ from $E'_{s+1}$ to $E'_s$ of degree prime to $\ell_{s+1}$. We have \begin{equation*}
   \alpha'\alpha \ell_{\rm 1}^{e_{\rm 1}}\cdots \ell_s^{e_s}\in \rm{I}({\rm{ker}}(\widehat{\beta_{\it s}\cdots\beta_{\rm 1}})).
 \end{equation*} Since $\rm{I}({\rm{ker}}(\widehat{\beta_{\it s}\cdots\beta_{\rm 1}}))=\rm{Hom}(\it{E}_{\rm 2},E'_s)\widehat{\beta_{s}\cdots\beta_{\rm 1}}$, it follows that
 \begin{equation*}
   \widehat\alpha'\alpha'\alpha \ell_{\rm 1}^{e_{\rm 1}}\cdots \ell_s^{e_s}\widehat{\beta_{s+1}}\in \rm{Hom}(\it E_{\rm 2},E'_{\it s+\rm 1})(\widehat{\beta_{\it s+\rm 1}\cdots\beta_{\rm 1}})
 \end{equation*} i.e.
 \begin{equation*}
   \rm{deg} \alpha' \ell_{\rm 1}^{e_{\rm 1}}\cdots \ell_s^{e_s}\ell_{s+1}^{\rho(e'_{s+1})}(b''+\prod\limits_{\it j=\rm 1}^{s}\ell_{\it j}^{\it e'_j}\mathit f \bar{\gamma})\in\mathbb Z \prod\limits_{\it j=\rm 1}^{s+1} \ell_{\it j}^{\it e_j}+\mathbb  Z \it t' \prod\limits_{\it j=\rm 1}^{s+1} \ell_{\it j}^{\rho(e'_{\it j})}(\it b'+\mathit f \bar{\gamma}).
 \end{equation*} Hence $t'\mid \rm{deg} \alpha' \ell_{\rm 1}^{\it e_{\rm 1}}\cdots \ell_{\it s}^{\it e_s}$. Thus $t'=1$.

\end{proof}

\begin{cor}
Given ordinary elliptic curves $E_{\rm 1}, E_{\rm 2}$ defined over $k$  where $[\rm{End}(\it{E}_{\rm{2}}):\rm{End}(\it{E}_{\rm{1}})]=\ell_{\rm 1}^{e'_{\rm 1}}\cdots \ell_s^{e'_s}$ for different primes $\ell_i$ and $e'_i\in \mathbb Z_{\ne0}$. Let $\beta: E_2 \to E_1$ be an isogeny without backtracking, suppose $\rm{End}(\it{E}_{\rm{2}})=\mathbb Z+\mathbb Z\mathit f\gamma$, then
\begin{equation*}
 \rm{Hom}(\it{E}_{\rm{1}},\it{E}_{\rm{2}})=\mathbb Z \widehat\beta+\mathbb Z( \prod\limits_{i=\rm 1}^s \ell_i^ {\rho(e'_i)}(b\prod\limits_{i=\rm 1}^s \ell_i^ {\rho(e'_i)-e'_i}+\mathit f\gamma)\widehat\beta)/(\rm{deg}\beta)
\end{equation*} for some integer $b$.
\end{cor}
\begin{proof}
Similar to the proof of Corollary 3.2.
\end{proof}

\begin{cor}
Given ordinary elliptic curves $E_{\rm 1}, E_{\rm 2}$ defined over $k$  where $[\rm{End}(\it{E}_{\rm{2}}):\rm{End}(\it{E}_{\rm{1}})]=\ell_{\rm 1}^{e'_{\rm 1}}\cdots \ell_s^{e'_s}$ for different primes $\ell_i$ and $e'_i\in \mathbb Z_{\ne0}$ Then every isogeny from $E_{\rm 2}$ to $E_{\rm 1}$ can correspond to an ideal of $\rm{End}(\it{E}_{\rm{2}})$ if and only if all $e'_i \leqslant 0$.
\end{cor}

\begin{proof}
Given any $\beta \in \rm{Hom}(\it{E}_{\rm{2}},\it{E}_{\rm{1}})$.  If all $e'_i \leqslant 0$, factor $\beta$ as $\beta_{r} \cdots \beta_{\rm 2}\beta_{\rm 1}$ for some integer $r$ and $\beta_1,\dots,\beta_r$ are isogenies of distinct prime power degrees. By  Corollary 3.3, every $\beta_i$ can correspond to an ideal, then $\beta$ can correspond to an ideal.

If there exists some $i$ such that $e'_i>0$. Factor $\beta$ as $\beta_2\beta_1$ where $\rm{deg}\beta_1$ is $p$-power and deg$\beta_2$ is prime to $p$. Then $\beta_2$ is separable. It can be proved similarly with the proof of Corollary 3.3 that $\beta$ can not correspond to an ideal.
\end{proof}

According to Theorem 3.1 and Corollary 3.5, whether an isogeny can correspond to a kernel ideal is only up to the two elliptic curves, more precisely, up to their endomorphism rings.

\section{The non-trivial minimal degree}
Let $\rm{Md}_{\it{k}}(\it{E}_{\rm{2}},\it{E}_{\rm{1}}):=\rm{min}\{deg\beta:\beta \in \rm{Hom}_{\it{k}}(\it{E}_{\rm{2}},\it{E}_{\rm{1}}), \rm{deg}\beta \ne 1\}$, $\rm{Md}_{\it{k}}(\it{E}):=\rm{Md}_{\it{k}}(\it{E},E)$,  $\rm{Md}(\it{E}_{\rm{2}},\it{E}_{\rm{1}}):=\rm{Md}_{\bar{\it{k}}}(\it{E}_{\rm{2}},\it{E}_{\rm{1}}), \rm{Md}(\it{E}):=\rm{Md}_{\bar{\it{k}}}(\it{E})$. Obviously, $\rm{Md}_{\it{k}}(\it{E}_{\rm{2}},\it{E}_{\rm{1}})=\rm{Md}_{\it{k}}(\it{E}_{\rm{1}},\it{E}_{\rm{2}})$. We call it the non-trivial minimal degree.

First we consider $\rm{Md}(\it{E})$. It suffices to check whether $E$ has an endomorphism of degree 2 or 3, since $2\in \rm{End}(\it{E})$. We need to check this with the help of the Deuring's lifting theorem\cite{deuring1941constructing,serge1987elliptic}.

\begin{thm}
Let E be an elliptic curve defined over a finite field and let $\alpha$ be an endomorphism of E. Then there exists an elliptic curve $\widetilde E$ defined over a finite extension H of $\mathbb Q$ and an endomorphism $\widetilde \alpha $ of   $\widetilde E$  s.t. E is the reduction of  $\widetilde E$ modulo some prime ideal of the ring of algebraic integers of $K$ and the reduction of  $\widetilde \alpha $  is $\alpha$.
\end{thm}

In fact, $\rm{deg}\widetilde \alpha=\rm{deg} \alpha$. We lift $\alpha$ to $\widetilde \alpha $.  Since $\widetilde\alpha \in \rm{End}(\widetilde E)$ has degree 2 or 3,  $\widetilde\alpha$  corresponds to an element of some imaginary quadratic order of norm 2 or 3 respectively. However, the number of the imaginary quadratic orders containing elements of norm 2 or 3 are finite. Those with norm 2 are $\mathbb Z[ \sqrt{-1}],\mathbb Z[\sqrt{-2}]$ and $\mathbb Z[\frac{1+\sqrt{-7}}{2}]$, and those with norm 3 are $\mathbb Z[\sqrt{-2}], \mathbb Z [\frac{1+\sqrt{-3}}{2}], \mathbb Z[\sqrt{-3}]$ and $\mathbb Z [\frac{1+\sqrt{-11}}{2}]$. Since $1+\sqrt{-2}\in \rm{End}(\it{E})$ implies $\sqrt{-2}\in \rm{End}(\it{E})$, then we take out $\mathbb Z[\sqrt{-2}]$ when considering whether $\rm{Md}(\it{E})=\rm 3$. For an imaginary quadratic order $\mathbb Z[\tau]$, let $j(\tau)$ denote the $j$-invariant of elliptic curve (over $\mathbb C$) having endomorphism ring by $\mathbb Z[\tau]$. Luckily, the above $j(\tau)$ of the above orders had been given out and there are elliptic curves over $\mathbb Q$ with the j-invariants\cite{cox2011primes,silverman2013advance}. Let $\Delta_E$ be the discriminant of elliptic curve $E$.

\begin{table}[H]
\centering
\begin{tabular}[{\ell}]{|c|c|c|c|}
\hline
$\tau$&$j(\tau)$ & \makecell {minimal Weierstrass\\ equation of $E$ over $\mathbb Q$}& $\Delta_E$\\
\hline
 $\sqrt{-1}$  & $12^3=1728$ & $y^2=x^3+x$  & $2^6$\\

 \hline
 $ \sqrt{-2}$  & $20^3$ & $y^2=x^3+4x^2+2x$  & $2^9$\\

 \hline
 $\frac{1+\sqrt{-7}}{2}$  & $-15^3$ & $y^2+xy=x^3-x^2-2x-1$  & $7^3$\\

 \hline
 $\frac{1+\sqrt{-3}}{2}$  & $0$ & $y^2+y=x^3$  & $3^3$\\

 \hline
 $\sqrt{-3}$  & $2^43^35^3$ & $y^2=x^3-15x+22$  & $2^83^3$\\

 \hline
 $\frac{1+\sqrt{-11}}{2}$  & $-32^3$ & $y^2+y=x^3-x^2-7x+10$  & $11^3$\\
 \hline

\end{tabular}\\
\footnotesize{Table 1}\\
\end{table}

Note that all $j(\tau)\in \mathbb Z$, then $j(\tau)\in \mathbb F_p$ after reduction. To make things more clearly, we still give out the results of supersingular case and ordinary case respectively. Coming back to the elliptic curves over finite fields, we also need Deuring's reduction theorem.

\begin{thm}
Let $E$ be an elliptic curve over a number field $H$ with $\rm{End}(\it{E}) \cong \mathcal O$, where $\mathcal O$ is an order of an imaginary quadratic field $K$. Let $\mathfrak P$ be prime ideal of $H$ over $p$, where $E$ has non-degenerate reduction $\bar E$. Then $\bar E$ is supersingular if and only if $p$ doesn't split in $K$.
\end{thm}
\begin{proof}
See \cite[Chapter 13, Theorem 12]{serge1987elliptic}.
\end{proof}

Hence we have

\begin{thm}
Let $E$ be a supersingular elliptic curve defined over $\overline{\mathbb F}_p$ with invariant $j$,  then $\rm{Md}(\it{E})=\rm{2}$ only when it satisfies one of the following conditions:\\
(1)\ $j=0$ when $p=2,3$;\\
(2)\ $j= 1728$ when $p\equiv 3\ (\rm{mod}\ 4)$;\\
(3)\ $j=2^65^3$ when $p\equiv 5,7\ (\rm{mod}\ 8)$;\\
(4)\ $j=-3^35^3$ when $p\equiv 3,5,6\ (\rm{mod}\ 7)$;\\
and $\rm{Md}(\it{E})=\rm{3}$ only when it satisfies one of the following conditions:\\
(5)\ $j=0$ when $p\equiv 2\ (\rm{mod}\ 3)$ and $p\ne2,5$;\\
(6)\ $j=2^43^35^3$ when $p\equiv 2\ (\rm{mod}\ 3)$ and $p\ne2,5,11,17,23$;\\
(7)\ $j=-2^{15}$ when $p\equiv 2,6,7,8 ,10\ (\rm{mod}\ 11)$ and $p\ne2,7,13,17,19$.\\
For other cases, $\rm{Md}(\it{E})=\rm4.$
\end{thm}
\begin{proof}
If $\alpha\in \rm{End}(\it{E})$ has degree 2, then by Theorem 4.1, $E$ and $\alpha$ can be lifted to $\widetilde E$ over some number field $H$ and $\widetilde \alpha$ where $\rm{deg} \widetilde \alpha=2$. Then $ \widetilde \alpha$ can be $1+\sqrt{-1}, \sqrt{-2}$ or $\frac{1+\sqrt{-7}}{2}$ where $\mathbb Z[\sqrt{-1}], \mathbb Z[\sqrt{-2}]$ and $\mathbb Z[\frac{1+\sqrt{-7}}{2}]$ are both maximal, hence $\rm{End}(\widetilde E)$ is $\mathbb Z[\widetilde \alpha]$. Thus $j_{\widetilde E}$ is $1728, 2^65^3$ or $-3^35^3$ respectively. Although $\widetilde E$ may not be defined over $\mathbb Q$, there are elliptic curves defined over $\mathbb Q$ with the same $j$-invariant as $\widetilde E$. Suppose $\widetilde E$ is defined over $\mathbb Q$, just as those listed in Table 1. Then by Theorem 4.2, for proper $p$, $\widetilde E$ has non-degenerate reduction modulo $p$ and becomes supersingular after reduction with the same $j$-invariant as $E$.

It is similar for case $\rm{Md}(\it{E})=\rm{3}$. Notice that $E$ may have both endomorphisms of degree 2 and 3, at that time $\rm{Md}(\it{E})=\rm{2}$, that is why $p$ can't be some values for the case $\rm{Md}(\it{E})=\rm{3}$.
\end{proof}

\begin{remark}: For the supersingular elliptic curves defined over $\mathbb F_{p^n}$ with not all endomorphisms defined over $\overline{\mathbb F}_{p^n}$. We can get $\rm{Md}_{\mathbb F_{p^n}}(\it{E})$ easily by some computations according to the possible conditions listed by Waterhouse \cite[Theorem 4.1]{waterhouse1969abe}. We don't list them here.
\end{remark}

\begin{thm}
Let $E$ be an ordinary elliptic curve defined over $\overline{\mathbb F}_p$ with invariant $j$,  then $\rm{Md}(\it{E})=\rm{2}$ only when it satisfies one of the following conditions:\\
(1)\ $j= 1728$ when $p\equiv 1\ (\rm{mod}\ 4)$;\\
(2)\ $j=2^65^3$ when $p\equiv 1,3\ (\rm{mod}\ 8)$;\\
(3)\ $j=-3^35^3$ when $p\equiv 1,2,4\ (\rm{mod}\ 7)$;\\
and $\rm{Md}(\it{E})=\rm{3}$ only when it satisfies one of the following conditions:\\
(4)\ $j=0$ when $p\equiv 1\ (\rm{mod}\ 3)$;\\
(5)\ $j=2^43^35^3$ when $p\equiv 1\ (\rm{mod}\ 3)$\\
(6)\ $j=-2^{15}$ when $p\equiv 1,3,4,5,9\ (\rm{mod}\ 11)$.\\
For other cases, $\rm{Md}(\it{E})=\rm4.$
\end{thm}

\begin{proof}
Similar to the proof of the supersingular case.
\end{proof}

Given two ordinary elliptic curves $E_{\rm 1},E_{\rm 2}$ defined over $k$ where $E_{\rm 1}\ncong E_{\rm 2}$ and $\rm{Hom}(\it{E}_{\rm{2}},\it{E}_{\rm{1}})\ne \rm0$. If $[\rm{End}(\it{E}_{\rm{2}}):\rm{End}(\it{E}_{\rm{1}})]=\ell^{e'}$ where $e'>0$, and suppose $E_{\rm 1}$ and $E_{\rm 2}$ are in the same $\ell$-volcano. If $E_{\rm 2}$ is above $E_{\rm 1}$ directly, then by Lemma 3.1, $\rm{Md}(\it{E}_{\rm{2}},\it{E}_{\rm{1}})=\ell^{e'}$. Otherwise, $\ell^{e'}\mid \rm{Md}(\it{E}_{\rm{2}},\it{E}_{\rm{1}})$ but they are not equal. Let $E_3$ be the elliptic curve directly above $E_{\rm 1}$ and $E_{\rm 2}$ with largest level. Suppose  $[\rm{End}(\it{E}_{\rm 3}):\rm{End}(\it{E}_{\rm{2}})]=\ell^{e''}$ where $e''>0$. The path from $E_{\rm 2}$ ascending to $E_3$ and then descending to $E_{\rm 1}$ corresponds to the isogeny of $\ell$-power degree from $E_{\rm 2}$ to $E_{\rm 1}$ without backtracking or endomorphism cycles and it is the shortest path from $E_2$ to $E_1$ in the $l$-volcano. Let $E'_{\rm 2}$ be the elliptic curve directly above $E_{\rm 1}$, if $\rm{Md}(\it{E}_{\rm{2}},\it{E}'_{\rm{2}})=\ell^{\rm2\it e''}$, then $\rm{Md}(\it{E}_{\rm{2}},\it{E}_{\rm{1}})=\ell^{e'+\rm2\it e''}$. But it does not always hold.

\begin{exam}
  Let $k=\mathbb{F}_{41}$, $t=6$, and $\ell=2$. For simplicity, the $k$-isomorphism classes of elliptic curves with trace $t$ are represented by their $j$-invariants in the following. Because the number of the $k$-isomorphism classes of elliptic curves with the same $j$-invariant is $2$ for the $j$-invariants not equal to 0 or $1728$ \cite[Theorem 2.2]{broker2006constructing} and the two $k$-isomorphism classes are twist of each other with their traces being opposite numbers. Here is the volcano we are considering.

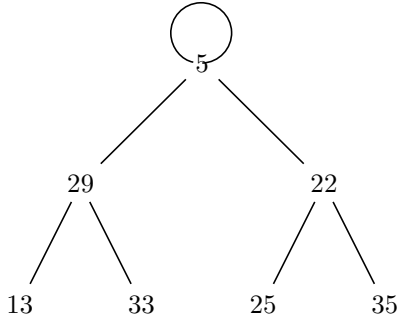
\begin{table}[H]
\centering
\begin{tikzpicture}[> = stealth, 
	shorten > = 1pt, 
	auto,
	node distance = 3cm, 
	semithick 
	,scale=.8,auto=left]
  \node (n1) at (0,0) {5};
  \node (n2) at (-2,-2) {29};
  \node (n3) at (2,-2) {22};
  \node (n4) at (-3,-4) {13};
  \node (n5) at (-1,-4) {33};
  \node (n6) at (1,-4) {25};
  \node (n7) at (3,-4) {35};
  \draw (n1) arc(-88:260:0.5);
  \draw (n1)--(n2);
  \draw (n1)--(n3);
   \draw (n2)--(n4);
    \draw (n2)--(n5);
     \draw (n3)--(n6);
      \draw (n3)--(n7);
\end{tikzpicture}
\end{table}
 Let $j(E_{\rm 2})=29$, $j(E_{\rm 1})=25$ and $j(E'_{\rm 2})=22$, then $[\rm{End}(\it{E}_{\rm{2}}):\rm{End}(\it{E}_{\rm{1}})]$ $=2$ and $\rm{Md}(\it{E}'_{\rm{2}},\it{E}_{\rm1})$ $=2 $. Since it can be checked by the modular equation \cite[Chapter 5, Theorem 5]{serge1987elliptic} of order $3$ that there is an isogeny from $E_{\rm 2}$ to $E'_{\rm 3}$ of degree $3$, then $\rm{Md}(\it{E}_{\rm{2}},\it{E}'_{\rm{2}})$ $=3\ne 2^2$. Hence  $\rm{Md}(\it{E}_{\rm{2}},\it{E}_{\rm1})$ $=6 $.
\end{exam}

More generally, if $[\rm{End}(\it{E}_{\rm{2}}): \rm{End}(\it{E}_{\rm{1}})]=$ $\ell^{e'}$ where $e'>0$, then there is $E'_2$ directly above $E_1$ such that $\rm{End}(\it{E}_{\rm{2}})\cong\rm{End}(\it{E}'_{\rm{2}})$. In this case,  $\rm{Md}(\it{E}_{\rm{2}},\it{E}_{\rm{1}})=\rm{Md}(\it{E}_{\rm{2}},\it{E}'_{\rm{2}})\ell^{e'}$ still holds. Thus this problem can be reduced to the case $\rm{End}(\it{E}_{\rm{2}})\cong\rm{End}(\it{E}'_{\rm{2}})$.
To get an upper bound for $\rm{Md}(\it{E}_{\rm{2}},\it{E}'_{\rm{2}})$, we need the following lemma.
\begin{lem}
  Let $\mathcal{O}$ be an imaginary quadratic order, then every invertible ideal class of $\mathcal{O}$ contains an ideal $J$ with
  \begin{equation*}
    \lVert J \rVert \leqslant \frac{2}{\pi} \sqrt{\vert \rm{disc}(\mathcal{O}) \vert}.
  \end{equation*}
\end{lem}
\begin{proof}
  See \cite[Chapter 5]{marcus1977number}. The proof is given there for the algebraic integer rings of number fields with the aid of Minkowski's theorem. It is easy to check that the proof still holds for imaginary quadratic orders. In fact, a proof for general orders of number rings is given by \cite[Theorem 5.4]{Stevenhagen2017number}.
\end{proof}

\begin{thm}
  Given two ordinary elliptic curves $E_1,E_2$ over $\mathbb{F}_{q}$ with trace $t$ and they are not $\mathbb{F}_{q}$-isomorphic. Then \begin{equation*}
         \rm{Md}(\it{E}_{\rm{2}},\it{E}_{\rm{1}})  \leqslant  \frac{\rm 2}{\pi}(\rm{4}\it q-t^{\rm 2})^{\frac{\rm1}{\rm2}}.
       \end{equation*}
\end{thm}
\begin{proof}
Suppose $\mathbb Z[\pi_{q,E_2}]=\mathbb Z +\mathbb Z\mathit f_0 \gamma$ with conductor $\mathit f_0=\prod\limits_{i=1}^{r}\ell_i^{e_i}$. And suppose $\rm{End}(\it E_{\rm1})$ $(\rm{End}(\it E_{\rm2}))$ has conductor $\mathit f_1=\prod\limits_{i=1}^{r}\ell_i^{e'_i}$ ($\mathit f_2=\prod\limits_{i=1}^{r}\ell_i^{e_i^{''}}$) where $0\leqslant e'_i (e_i^{''}) \leqslant e_i$ for all $i$. Then there are $E'_1, E'_2$ whose endomorphism rings have conductor $f=\prod\limits_{i=1}^{r}\ell_i^{\rm{min}(\it e'_i,e_i^{''})}$ such that $\rm{Md}(\it{E}_{\rm 1},\it{E}'_{\rm 1})=\prod\limits_{i=1}^{r}\ell_i^{e'
_i-\rm{min}(\it e'_i,e_i^{''}) }$ and $\rm{Md}(\it{E}_{\rm 2},\it{E}'_{\rm 2})=\prod\limits_{i=1}^{r}\ell_i^{e^{''}_i-\rm{min}(\it e'_i,e^{''}_i) }$.
Hence
\begin{equation*}
\begin{split}
    \rm{Md}(\it{E}_{\rm 1},\it{E}_{\rm 2}) & \leqslant \rm{Md}(\it{E}'_{\rm 1},\it{E}'_{\rm 2})\rm{Md}(\it{E}_{\rm 1},\it{E}'_{\rm 1})\rm{Md}(\it{E}_{\rm 2},\it{E}'_{\rm 2}) \\
     & \leqslant \frac{2}{\pi} \sqrt{\frac{\rm 4 \it q-t^{\rm 2}}{\prod\limits_{i=1}^{r}\ell_i^{2(e_i-\rm{min}(\it e'_i,e^{''}_i))}}}(\prod\limits_{i=1}^{r}\ell_i^{e'
_i-\rm{min}(\it e'_i,e_i{''}) })( \prod\limits_{i=1}^{r}\ell_i^{e^{''}_i-\rm{min}(\it e'_i,e^{''}_i) }) \   \rm (by\ Lemma 4.1)\\
     & = \frac{2}{\pi}\sqrt{(\rm 4 \it q-t^{\rm 2})\prod\limits_{i=1}^{r}\ell_i^{\rm{2}(\it e'_i+e^{''}_i-e_i-\rm{min}(\it e'_i,e^{''}_i))}}\\
     & \leqslant \frac{\rm 2}{\pi}(\rm{4}\it q-t^{\rm 2})^{\frac{\rm1}{\rm2}}.
\end{split}
\end{equation*}
\end{proof}
The result in Theorem 4.5 is rough. For two elliptic curves with known endomorphism rings, the result can be more precise.
\begin{exam}
 Let eB$(k,t)$ denote the upper bound given by Theorem 4.5, and rB$(k,t)$ denote the  largest non-trivial minimal degree in the $k$-isogenous class with trace $t$.

  As in Example 1, let $k=\mathbb{F}_{41}$, $t=6$, we have eB$(\mathbb{F}_{41},6)=7$. The $3$-isogeny garph is
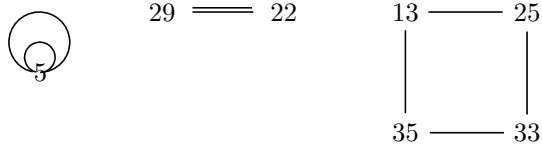
\begin{table}[H]
\centering
\begin{tikzpicture}[> = stealth, 
	shorten > =0.5pt, 
	auto,
	node distance = 3cm, 
	semithick 
	,scale=.8,auto=left]
  \node (n1) at (0,-1) {5};
  \node (n2) at (2,0) {29};
  \node (n3) at (4,0) {22};
  \node (n4) at (6,0) {13};
  \node (n5) at (8,-2) {33};
  \node (n6) at (8,0) {25};
  \node (n7) at (6,-2) {35};
  \draw (n1) arc(-88:260:0.5);
  \draw (n1) arc(-88:260:0.25);
  \draw (2.5,0)-- (3.5,0) [yshift=2pt]  (2.5,0) --  (3.5,0);
  \draw (n4)--(n6);
   \draw (n4)--(n7);
     \draw (n5)--(n6);
      \draw (n5)--(n7);
\end{tikzpicture}
\end{table}
Then rB$(\mathbb{F}_{41},6)=6$.

  Similarly, let $k=\mathbb{F}_{53}$, $t=-4$, we have  eB$(\mathbb{F}_{53},-4)=8$ and rB$(\mathbb{F}_{53},-4)=7$. Let $k=\mathbb{F}_{67}$, $t=12$, we have  eB$(\mathbb{F}_{67},12)=7$ and rB$(\mathbb{F}_{67},12)=5$.

\end{exam}

\begin{thm}
  Given two supersingular elliptic curves $E_{\rm 1},$ $E_{\rm 2}$ over $\mathbb{F}_{p^2}$. Then $\rm{Md}(\it{E}_{\rm{2}},\it{E}_{\rm{1}}) $ $\sim O(p)$ and $\rm{Md}_{\mathbb{F}_{\it p^{\rm 2}}}(\it{E}_{\rm{2}},\it{E}_{\rm{1}}) $ $\sim O(p)$.
   If $E_{\rm 1},$ $E_{\rm 2}$ can be defined over $\mathbb{F}_{p}$ and they are not $\mathbb{F}_{p}$-isomorphic, then $\rm{Md}_{\mathbb{F}_{\it p}}(\it{E}_{\rm{2}},\it{E}_{\rm{1}}) $ $\leqslant  \frac{4}{\pi}p^{\frac{1}{2}}$.
\end{thm}
\begin{proof}
  As a result of \cite[Theorem 79]{kohel1996endomorphism}, $\rm{Md}(\it{E}_{\rm{2}},\it{E}_{\rm{1}}) $ $\sim O(p)$. If the trace is $\pm2p$, then $\rm{Md}_{\mathbb{F}_{\it p^{\rm 2}}}(\it{E}_{\rm{2}},\it{E}_{\rm{1}})$ $=\rm{Md}(\it{E}_{\rm{2}},\it{E}_{\rm{1}})$. If the trace is not $\pm2p$ and they are not $\mathbb{F}_{p^2}$-isomorphic, then $\rm{Md}_{\mathbb{F}_{\it p^{\rm 2}}}(\it{E}_{\rm{2}},\it{E}_{\rm{1}})=p$\cite{adj2019iso}. If $E_{\rm 1},$ $E_{\rm 2}$ can be defined over $\mathbb{F}_{p}$, by Lemma 4.1, we get $\rm{Md}_{\mathbb{F}_{\it p}}(\it{E}_{\rm{2}},\it{E}_{\rm{1}}) $ $\leqslant  \frac{4}{\pi}p^{\frac{1}{2}}$.
\end{proof}
\begin{exam}
Let $p=53$, then eB$(\mathbb{F}_{53},0)=9$. There is only one order $\mathcal O$ with discriminant $-2^253$ and it has ring class number 6. There are totally six $\mathbb{F}_{53}$-isomorphism classes and one $j$-invariant correspond to two $\mathbb{F}_{53}$-isomorphism classes. Denote them by $E(0,1),E(0,2),E(-3,1),E(-3,2)$, $E(-7,1),$ $E(-7,2)$. The $2$-isogeny graph is
\begin{table}[H]
\centering
\begin{tikzpicture}[> = stealth, 
	shorten > = 1pt, 
	auto,
	node distance = 3cm, 
	semithick 
	,scale=.8,auto=left]
  \node (n1) at (0,0) {E(0,1)};
  \node (n2) at (2,0) {E(-7,1)};
  \node (n3) at (4,0) {E(-3,1)};
  \node (n4) at (6,0) {E(-3,2)};
  \node (n5) at (8,0) {E(0,2)};
  \node (n6) at (10,0) {E(-7,2)};

  \draw (n1)--(n2);
   \draw (n3)--(n4);
     \draw (n5)--(n6);
\end{tikzpicture}
\end{table}
And the $3$-isogeny graph is
\begin{table}[H]
\centering
\begin{tikzpicture}[> = stealth, 
	shorten > = 1pt, 
	auto,
	node distance = 3cm, 
	semithick 
	,scale=.8,auto=left]
  \node (n1) at (0,0) {E(0,1)};
  \node (n2) at (2,0) {E(0,2)};
  \node (n3) at (4,0) {E(-3,2)};
  \node (n4) at (0,-1.5) {E(-3,1)};
  \node (n5) at (2,-1.5) {E(-7,2)};
  \node (n6) at (4,-1.5) {E(-7,1)};
  \draw (n1)--(n2);
   \draw (n2)--(n3);
   \draw (n3)--(n6);
     \draw (n4)--(n5);
      \draw (n5)--(n6);
      \draw (n1)--(n4);
\end{tikzpicture}
\end{table}
Then rB$(\mathbb{F}_{53},0)=6$. The five non-principal ideal classes in ring class group $cl(\mathcal{O})$ can be represented by one ideal of norm 2, two ideals of norm 3 and two ideals of norm 6.
\end{exam}

\section*{References}
\bibliographystyle{plain}
\bibliography{bibfilelorry}

\begin{thebibliography}{10}

\bibitem{adj2019iso}
Gora Adj, Omran Ahmadi, and Alfred Menezes.
\newblock On isogeny graphs of supersingular elliptic curves over finite
  fields.
\newblock {\em Finite Fields and Their Applications}, 55:268--283, 2019.

\bibitem{azarderakhsh2017supersingular}
Reza Azarderakhsh, Matthew Campagna, Craig Costello, LD~Feo, Basil Hess, Amir
  Jalali, David Jao, Brian Koziel, Brian LaMacchia, Patrick Longa, et~al.
\newblock Supersingular isogeny key encapsulation.
\newblock {\em Submission to the NIST Post-Quantum Standardization project},
  2017.

\bibitem{broker2006constructing}
Reinier Br{\"o}ker.
\newblock {\em Constructing elliptic curves of prescribed order}.
\newblock PhD thesis, Universiteit Leiden, 2006.

\bibitem{buchmann2007binary}
Johannes Buchmann and Ulrich Vollmer.
\newblock {\em Binary Quadratic Forms: An Algorithmic Approach}, volume~20.
\newblock Springer, Berlin, 2007.

\bibitem{charles2009cry}
Denis~X Charles, Kristin~E Lauter, and Eyal~Z Goren.
\newblock Cryptographic hash functions from expander graphs.
\newblock {\em Journal of Cryptology}, 22(1):93--113, 2009.

\bibitem{cox2011primes}
David~A Cox.
\newblock {\em Primes of the form $x^2+ ny^2$: Fermat, class field theory, and
  complex multiplication}, volume~34.
\newblock John Wiley \& Sons, 1989.

\bibitem{deuring1941constructing}
Max Deuring.
\newblock Die typen der multiplikatorenringe elliptischer funktionenkörper.
\newblock {\em In Abhandlungen aus dem mathematischen Seminar der Universität
  Hamburg}, 14(1):197--272, 1941.

\bibitem{galbraith2017identification}
Steven~D Galbraith, Christophe Petit, and Javier Silva.
\newblock Identification protocols and signature schemes based on supersingular
  isogeny problems.
\newblock In {\em International Conference on the Theory and Application of
  Cryptology and Information Security}, pages 3--33. Springer, 2017.

\bibitem{kohel1996endomorphism}
David~Russell Kohel.
\newblock {\em Endomorphism rings of elliptic curves over finite fields}.
\newblock PhD thesis, University of California, Berkeley, 1996.

\bibitem{serge1987elliptic}
Serge Lang.
\newblock {\em Elliptic functions}.
\newblock Springer, New York, 1987.

\bibitem{marcus1977number}
Daniel~A Marcus.
\newblock {\em Number fields}.
\newblock Springer, New York, 1977.

\bibitem{pizer1990ram}
Arnold~K Pizer.
\newblock Ramanujan graphs and hecke operators.
\newblock {\em Bulletin of the American Mathematical Society}, 23(1):127--137,
  1990.

\bibitem{schoff1987nonsingular}
René Schoof.
\newblock Nonsingular plane cubic curves over finite fields.
\newblock {\em Journal of combinatorial theory}, (Series A, 46(2)):183--211,
  1987.

\bibitem{silverman2009arithmetic}
Joseph~H Silverman.
\newblock {\em The arithmetic of elliptic curves}, volume 106.
\newblock Springer Science \& Business Media, 2009.

\bibitem{silverman2013advance}
Joseph~H Silverman.
\newblock {\em Advanced topics in the arithmetic of elliptic curves}.
\newblock Springer Science \& Business Media, 2013.

\bibitem{Stevenhagen2017number}
P.~Stevenhagen.
\newblock Number rings.
\newblock \url{http://websites.math.leidenuniv. nl/algebra/ant.pdf}, Oct. 2017.

\bibitem{sutherland2013isogeny}
Andrew Sutherland.
\newblock Isogeny volcanoes.
\newblock {\em The Open Book Series}, 1(1):507--530, 2013.

\bibitem{tate1966end}
John Tate.
\newblock Endomorphisms of abelian varieties over finite fields.
\newblock {\em Inventiones mathematicae}, 2(2):134--144, 1966.

\bibitem{voight2019quaternion}
John Voight.
\newblock Quaternion algebras.
\newblock \url{https://math.dartmouth.edu/~jvoight/quat-book.pdf}, 2019.
\newblock Version v.0.9.15.

\bibitem{washington2008elliptic}
Lawrence~C. Washington.
\newblock {\em Elliptic Curves: Number Theory and Cryptography, Second
  Edition}.
\newblock New York, 2008.

\bibitem{waterhouse1969abe}
William~C Waterhouse.
\newblock Abelian varieties over finite fields.
\newblock In {\em Annales scientifiques de l'{\'E}cole Normale Sup{\'e}rieure},
  volume~2, pages 521--560, 1969.

\end{thebibliography}
\section*{Appendix A. The discussion about the index problem in field with characteristic 0.}
Consider two elliptic curves $E_1,E_2$ defined over a perfect field $F$ with characteristic 0, then the endomorphism rings of them are $\mathbb Z$ or a imaginary quadratic order\cite[Corollary \uppercase\expandafter{\romannumeral3}.9.4]{silverman2009arithmetic}. If Hom$(E_1,E_2)\ne 0$ and rank$_{\mathbb{Z}}(\rm{End}\it E_{\rm2})=\rm 1$, then rank$_{\mathbb{Z}}(\rm{End}\it E_{\rm1})=\rm 1$ and Hom$(E_1,E_2)=\mathbb Z \beta_0$  where $\beta_0\in \rm{Hom}(\it E_{\rm 1}, E_{\rm 2})$. Hence for any isogeny $\beta:E_2\to E_1$, we have
\begin{equation*}
  [\rm{End}(\it E_{\rm 2}):\rm{Hom}(\it E_{\rm 1}, E_{\rm 2})\beta]=(\frac{\rm{deg}\beta}{\rm{deg}\beta_0})^{\frac{\rm1}{\rm2}}\rm{deg}\beta_0.
\end{equation*}
If Hom$(E_1,E_2)\ne 0$ and rank$_{\mathbb{Z}}(\rm{End}\it E_{\rm2})=\rm 2$, then rank$_{\mathbb{Z}}(\rm{End}\it E_{\rm1})=\rm 2$ and End$(E_1)$ has the same endomorphism algebra as End$(E_2)$. If $E_1, E_2$ are defined over some number field, then for any isogeny $\beta:E_2\to E_1$, the index of $\rm{Hom}(\it E_{\rm 1}, E_{\rm 2})\beta$ in $\rm{End}(\it E_{\rm 2})$ has the same results with the ordinary case of elliptic curves defined over a finite field $k$. That is, if $\rm{End}(\it E_{\rm2})=\mathbb Z+\mathbb{Z}f\gamma$ and $[\rm{End}(\it{E}_{\rm{2}}):\rm{End}(\it{E}_{\rm{1}})]=$ $\ell_{\rm 1}^{e'_{\rm 1}}\cdots \ell_s^{e'_s}$ where $e'_i \in \mathbb{Z}_{\ne0}$ for all $i$, then
\begin{equation*}
[\rm{End}(\it{E}_{\rm{2}}):\rm{Hom}(\it{E}_{\rm{1}},\it{E}_{\rm{2}})\beta]=\prod\limits_{\it{i}=\rm 1}^s \ell_i^ {\rho(e'_i)}\rm{deg} \beta.
\end{equation*}

 For an elliptic curve $E$ defined over $\mathbb C$, then there is a lattice $L$ in the form $\mathbb Z+ \mathbb Z \tau$ where $\tau$ lies in the upper half plane
  \begin{equation*}
    \mathcal{H}=\{x+iy\in \mathbb C: y>0\},\  \rm{where}\   \it i^{\rm2}=-\rm1,
  \end{equation*}
such that $E\cong\mathbb C/L$. Given such two lattices $L_1=\mathbb Z+ \mathbb Z \tau_1$ and $L_2=\mathbb Z+ \mathbb Z \tau_2$, consider the isogenies from $\mathbb C/L_1$ to $\mathbb C/L_2$, according to \cite[Theorem \uppercase\expandafter{\romannumeral6}.5.3]{silverman2009arithmetic}, it is equivalent to consider
\begin{equation*}
  \rm{Map}(\it L_{\rm 1},L_{\rm 2})=\{\alpha\in \mathbb C: \alpha L_{\rm 1}\subseteq L_{\rm 2}\}.
\end{equation*}
Let $K$ be an imaginary quadratic field, by some computations we have the following facts:
\begin{enumerate}[1)]
  \item $\mathbb C/(\mathbb Z+ \mathbb Z \tau)$ has the endomorphism ring isomorphic to an order in $K$ if and only if $\tau \in K$.
  \item If $\tau_1\in K$, then Hom$(\mathbb C/L_1, \mathbb C/L_2)\ne 0$ if and only if $\tau_2\in K$.
\end{enumerate}

If $\tau_1$ and $\tau_2$ are not contained in any imaginary quadratic fields, and Hom$(\mathbb C/L_1, \mathbb C/L_2)\ne 0$, then $\mathbb{Q}(\tau_1)=\mathbb{Q}(\tau_2)$. If $\tau_1=a_0+a_1\tau_2$ for some $a,b \in \mathbb Q$, then $\rm{Map}(\it L_{\rm 1},L_{\rm 2})=\mathbb Z m$ form some integer $m$. Otherwise, suppose $[\mathbb{Q}(\tau_2):\mathbb{Q}]=n$, let $\tau_1$ be in the form of $a_0+a_1\tau_2+\cdots+a_{n-1}\tau_2^{n-1}$ and the minimal polynomial of $\tau_2$ be $b_0+b_1x+\cdots+b_{n-1}x^{n-1}+x^n$ where $a_i,b_i\in \mathbb{Q}$ for all $i$, then $\rm{Map}(\it L_{\rm 1},L_{\rm 2})=\mathbb Z\beta $ where $\beta\notin \mathbb Z$ if and only if $ a_{n-1}\ne 0$ and
\begin{equation*}
  \left(
  \begin{array}{cc}
    a_2 & a_1+a_{n-1}b_2 \\
    a_3 & a_2+a_{n-1}b_3 \\
    \vdots & \vdots \\
    a_{n-1}& a_{n-2}+a_{n-1}b_{n-1}
  \end{array}
  \right)
\end{equation*}
has rank $1$.
\section*{Appendix B. An application of Corollary 3.5.}
Let $K$ be an imaginary quadratic field, and $ \mathbb Z [\gamma]$ be its algebraic integer ring. For simplicity, we assume $\gamma\ne \sqrt{-1}, \frac{1+\sqrt{-3}}{2}$  . Arbitrarily given a conductor $f$, for any $\ell \mid f$, we can enumerate the number of invertible or non-invertible ideals of $\ell$-power norms. We will recur to the Corollary 3.5. Since the endomorphism ring of $\mathbb C/(\mathbb Z [f\gamma])$ is $\mathbb Z[ f\gamma]$, then after doing reduction by proper prime ideal , we can get an elliptic curve $E$ defined over a finite field with endomorphism ring $\mathbb Z [f\gamma]$. Since distinct ideals of $\mathbb Z [f\gamma]$ correspond to distinct isogenies from $E$, to enumerate the number of invertible or non-invertible ideals of $\ell$-power norms, it suffices to enumerate the number of the corresponding isogenies in the $\ell$-volcano. Let $D$ be the discriminant of $K$,  iG$(f,\ell^n)$ denote the number of invertible ideals of norm $\ell^n$, and niG$(f,\ell^n)$ denote the number of non-invertible ideals of norm $\ell^n$ in the order with conductor $f$. Then we have the following results.
For the invertible ideals, we have
\begin{equation*}
\rm{iG}(\it f, \ell^n)=\left\{
\begin{array}{ll}
  0&if\ n\ is\ odd\ and\ (\romannumeral1)0<n<2v_{\ell}(f)\ or\ (\romannumeral2)   (\frac{D}{\ell})=-1; \\
  \ell^{\frac{n}{2}} & \ if\ n\ is\ even\ and\ 0< n<2v_{\ell}(f); \\
  \ell^{v_{\ell}(f)-1}(\ell+1) & \ if\ n\geqslant2v_{\ell}(f),\ n\ is\ even\ and\ (\frac{D}{\ell})=-1; \\
  \ell^{v_{\ell}(f)} &\ if\ n\geqslant2v_{\ell}(f)\ and\ (\frac{D}{\ell})=0; \\
  (n-2v_{\ell}(f)+1)\ell^{v_{\ell}(f)-1}(\ell-1) & \ if\ n\geqslant2v_{\ell}(f)\ and\  (\frac{D}{\ell})=1.
\end{array}
\right.
\end{equation*}
 And for the non-invertible ideals, we have
\begin{equation*}
  \rm{niG}(\it f, \ell^n)=\sum\limits_{1\leqslant k\leqslant \rm{min}(\it n,v_{\ell}(f))}\rm{iG}(\it \frac{f}{\ell^k}, \ell^{n-k}),\  where\ \rm{iG}(1,1)=1.
\end{equation*}
\begin{exam}
  Consider the order $\mathbb Z[2^2\sqrt{-2}]$. We list the ideals and the numbers in the following table.
  \begin{table}[H]
\centering
\begin{tabular}[{\ell}]{|c|c|c|c|c|}
\hline
\text{norm}&\text{invertible ideals} & &\text{non-invertible ideals} & \\
\hline
 $2$  &  &$\rm{iG}(2^2, 2)=0$  & $2\mathbb Z +2^2\sqrt{-2}\mathbb Z$ & $\rm{niG}( 2^2, 2)=1$\\

 \hline
 $ 2^2$  & \makecell{$2^2\mathbb Z +(2+2^2\sqrt{-2})\mathbb Z$\\$2\mathbb Z +2^3\sqrt{-2}\mathbb Z$} & $\rm{iG}(2^2, 2^2)=2$  & \makecell{$ 2^2\mathbb Z+2^2\sqrt{-2}\mathbb Z$} & \makecell{$\rm{iG}(2,2)+\rm{iG}(1,1)$\\=0+1=1} \\

 \hline
 $2^3$  & &$\rm{iG}(2^2, 2^3)=0$  &\makecell{$ 2^2\mathbb Z+2^3\sqrt{-2}\mathbb Z$\\$ 2^3\mathbb Z+(2^2+2^2\sqrt{-2})\mathbb Z$\\$2^3\mathbb Z +2^2\sqrt{-2}\mathbb Z$}& \makecell{$\rm{iG}(2,2^2)+\rm{iG}(1,2)$\\=2+1=3}\\

 \hline
 $2^4$  &\makecell{$2^2\mathbb Z +2^4\sqrt{-2}\mathbb Z$\\$2^3\mathbb Z +(2^2+2^3\sqrt{-2})\mathbb Z$\\$ 2^4\mathbb Z+\mathbb Z(2^2+2^2\sqrt{-2})$\\$2^4\mathbb Z +(2^23+2^2\sqrt{-2})\mathbb Z$} & $\rm{iG}(2^2, 2^4)=4$  & \makecell{$2^3\mathbb Z +2^3\sqrt{-2}\mathbb Z$\\$ 2^4\mathbb Z+2^2\sqrt{-2}\mathbb Z$\\$2^4\mathbb Z +(2^3+2^2\sqrt{-2})\mathbb Z$}& \makecell{$\rm{iG}(2,2^3)+\rm{iG}(1,2^2)$\\=2+1=3}\\

 \hline
 $2^5$  & \makecell{$2^5\mathbb Z +2^2\sqrt{-2}\mathbb Z$\\$2^5\mathbb Z +(2^3+2^2\sqrt{-2})\mathbb Z$\\$2^5\mathbb Z +(2^4+2^2\sqrt{-2})\mathbb Z$\\$2^5\mathbb Z +(2^33+2^2\sqrt{-2})\mathbb Z$} & $\rm{iG}(2^2, 2^5)=4$  & \makecell{$2^3\mathbb Z +2^4\sqrt{-2}\mathbb Z$\\$2^4\mathbb Z +(2^3+2^3\sqrt{-2})\mathbb Z$\\$2^4\mathbb Z +2^3\sqrt{-2}\mathbb Z$} & \makecell{$\rm{iG}(2,2^4)+\rm{iG}(1,2^3)$\\=2+1=3}\\

 \hline

\end{tabular}\\

\end{table}

\end{exam}
\end{document}